\documentclass[11pt, leqno]{amsart}

\usepackage{amsmath}
\usepackage{amsfonts}
\usepackage{amssymb}
\usepackage{amsthm}
\usepackage{enumitem}
\usepackage{thm-restate}
\usepackage{mathtools}

\usepackage{hyperref}
\hypersetup{hidelinks}

\usepackage{geometry}
\newtheorem{theorem}{Theorem}[section]
\newtheorem{lemma}{Lemma}[section]
\newtheorem{corollary}[lemma]{Corollary}

\theoremstyle{remark}
\newtheorem*{remark}{Remark}

\theoremstyle{definition}

\numberwithin{equation}{section}

\DeclareMathOperator{\sech}{sech}
\DeclareMathOperator{\Ai}{Ai}
\DeclareMathOperator{\artanh}{artanh}

\DeclareMathOperator{\supp}{supp}
\newcommand{\Mod}[1]{\ (\mathrm{mod}\ #1)}

\title{The Second Moment of Sums of Hecke Eigenvalues I}
\author{Ned Carmichael}
\date{March 10, 2026}
\address{Department of Mathematics, King’s College London, London, WC2R 2LS, UK}
\email{ned.carmichael@kcl.ac.uk}

\begin{document}

\begin{abstract}
Let \(f\) be a holomorphic Hecke cusp form of weight \(k\) for \(\mathrm{SL}_2(\mathbb{Z})\), and let \((\lambda_f(n))_{n\geq1}\) denote its sequence of Hecke eigenvalues. We compute the first and second moments of the sums \(\mathcal S(x,f)=\sum_{x\leq n\leq 2x}\lambda_f(n)\), on average over forms \(f\) of large weight \(k\), in the regime where the length of the sums \(x\) is smaller than \(k^2\). We observe transitions in the size of the sums when \(x\approx k\) and \(x\approx k^2\). In subsequent work \cite{paper2}, it will be shown that once \(x\) is larger than \(k^2\) (where the latter transition occurs), the average size of the sums \(\mathcal S(x,f)\) becomes dramatically smaller. 
\end{abstract}

\maketitle

\section{Introduction}

Understanding the distribution of sums of arithmetical functions is a typical problem in analytic number theory, the Dirichlet divisor and Gauss circle problems being two classical examples. In this paper, we are interested in sums of Hecke eigenvalues attached to cusp forms. Let \(\mathcal B_k\) denote a basis for the space of holomorphic cusp forms of even weight \(k\) for the full modular group \(\mathrm{SL}_2(\mathbb{Z})\), consisting of orthogonal Hecke eigenforms. We write the Fourier expansion of a given form \(f\in\mathcal B_k\) as 
\[f(z)=\sum_{n\geq1}\lambda_f(n)n^{(k-1)/2} e^{2\pi i nz},\]
and normalise so that \(\lambda_f(1)=1\).  Since \(f\) is an eigenform (one has \(T_n f=\lambda_f(n)n^{(k-1)/2} f\) for the Hecke operators \((T_n)_{n\geq1}\)), the eigenvalues \((\lambda_f(n))_{n\geq1}\) are real, and Deligne's bound states \(|\lambda_f(n)|\leq d(n)\) (where \(d(n)\) denotes the divisor function). 

We will study the sums of eigenvalues
\begin{equation*}
\mathcal S(x,f)\coloneqq\sum_{x<n\leq 2x} \lambda_f(n),
\end{equation*}
as \(f\) varies over the basis \(\mathcal B_k\) of weight \(k\) eigenforms, with \(k\) being large (and even). To this end, we introduce the averaging operator 
\begin{equation}\label{favs} 
\langle g(f) \rangle=\sum_{f\in \mathcal B_k}\omega(f)g(f),
\end{equation}
where \(\omega(f)\) are the \emph{harmonic weights}
\[\omega(f)=\frac{\Gamma(k-1)}{(4\pi)^{k-1}\|f\|^2}.\]
These weights account for the fact that each \(f\in\mathcal B_k\) is arithmetically normalised (in the sense \(\lambda_f(1)=1\)), rather than \(L^2\) normalised. Note that \(\langle 1\rangle =\sum_f \omega(f)=1+\mathcal O(e^{-k})\). Our results describe the behaviour of the first and second moments of the sums \(\mathcal S(x,f)\) -- that is, the averages
\[\langle \mathcal S(x,f)\rangle = \sum_{f\in\mathcal B_k}\omega(f) \mathcal S(x,f) \: \text{ and } \: \langle \mathcal S(x,f)^2 \rangle = \sum_{f\in\mathcal B_k}\omega(f) \mathcal S(x,f)^2.\]

\begin{theorem}\label{thmmean}
We have the following estimates for the first moment of the sums \(\mathcal S(x,f)\).
\begin{enumerate}[label=(\roman*)]
    \item \label{meani} If \(x\leq k^2/(32\pi^2+1)\), then
    \[\langle \mathcal S(x,f)\rangle \ll e^{-\sqrt k}.\]
    \item \label{meanii} If \(k^2/(32\pi^2-1)\leq x\leq k^2/(16\pi^2+1)\), then
    \[\langle \mathcal S(x,f)\rangle = \frac{(-1)^{k/2}}{4\pi}k+\mathcal O(k^{7/8}).\]
\end{enumerate}
\end{theorem}

We also compute the second moment \(\langle \mathcal S(x,f)^2\rangle\). Before giving the asymptotics, we must first introduce a (continuous) function \(L:[0,\infty)\to\mathbb{R}\), given by
\begin{equation}\label{ldef}
L(\xi)\coloneqq \begin{cases} 0 & \text{ if } 0\leq \xi\leq 1,\\ \log\xi & \text{ if } 1\leq \xi\leq \sqrt{2},\\ \log(2/\xi) & \text{ if } \sqrt{2}\leq \xi\leq 2,\\ 0& \text{ if } 2\leq \xi.\end{cases}
\end{equation}

\begin{theorem}\label{thmvarsimple}
Assume \(x\to\infty\) with \(k\), but \(x=o(k^2/\log^6 k)\). Then 
\[\langle \mathcal S(x,f)^2\rangle \sim x+\frac{(-1)^{k/2}}{2\pi}L\Big(\frac{k}{4\pi x}\Big)k.\]
\end{theorem}

A more precise statement, including error terms, is given in Theorem \ref{thmvar}. Note that the secondary term in Theorem \ref{thmvarsimple} appears only if \(k/(8\pi)\leq x\leq k/(4\pi)\), but in this range one still has \(\langle \mathcal S(x,f)^2\rangle \asymp x\).

Theorems \ref{thmmean} and \ref{thmvarsimple} display transitions in the sizes of the first and second moments, appearing when \(x\in [k^2/(32\pi^2), k^2/(16\pi^2)]\) and \(x\in[k/(8\pi), k/(4\pi)]\) respectively. These are a consequence of the transitions of Bessel functions which arise naturally in the computation.

After the transition occurring when \(x\in[k^2/(32\pi^2), k^2/(16\pi^2)]\), the behaviour of the sums changes significantly. The average size of the sums shifts from being around \(x^{1/2}\) when roughly \(x\leq k^2/(16\pi^2)\), to around \(x^{1/4}\) when roughly \(x\geq k^2/(16\pi^2)\). The first and second moments when \(x\geq k^2/(8\pi^2)\) are studied in \cite{paper2}. 
The precise asymptotic behaviour is unclear around the transition. 
It is an interesting problem to extend our results to address this regime. 
%In principle our method could be used here, but one would encounter severe technical difficulties.

Bober, Booker, Lee and Lowry-Duda \cite{bblld} have studied sums of the eigenvalues \((\lambda_f(n))_{n\geq1}\) taken over primes (also in weight aspect). Interestingly, they observe `murmurations' in the mean values of these sums when their length is \(\approx k^2\). See also Zubrilina \cite{zubrilina}, which provides another example of the murmurations phenomenon in the level aspect.

Transitions analogous to those we observe also appear in different settings. 
For example, the distribution of sums of eigenvalues in arithmetic progressions has been studied by several authors (on average over the congruence class). 
For a fixed eigenform \(f\), Blomer \cite{blomer} first considered the variance 
\begin{equation}\label{blomervar}
\mathrm{Var}\coloneqq \sum_{b\Mod q} \Big|\sum_{\substack{n\leq x\\n\equiv b\Mod q}}\lambda_f(n)\Big|^2,
\end{equation}
and showed \(\mathrm{Var}\ll_{f,\epsilon} x^{1+\epsilon}\) for any \(\epsilon>0\). L\"u \cite{lu} sharpened this, removing the \(\epsilon\), and improved the bound to \(\mathrm{Var}\ll_{f,\epsilon} q^{4/3}x^{2/3+\epsilon}\) in the range \(x\geq q^4\).

Lau and Zhao \cite{lauzhao} later proved asymptotics for the variance (\ref{blomervar}) in certain ranges of \(x\). Notably, they observed a transition in the behaviour of the sums when \(x\approx q^2\). For \(x\geq q^{4-\epsilon}\) they improve L\"u's bound to \(\mathrm{Var}\ll_{f,\epsilon} q^{1/3}x^{2/3+\epsilon}\). However, for \(q^{2+\epsilon}\leq x\leq q^{4-\epsilon}\), they are able to prove a (complicated) asymptotic for \(\mathrm{Var}\). For simplicity, we do not state the asymptotic here, but note that it implies \(qx^{1/2}/\log\log q \ll_f \mathrm{Var}\ll_f qx^{1/2}\). On the other hand, they prove that for \(q^{1+\epsilon}\leq x\leq q^{2-\epsilon}\), one has \(\mathrm{Var}\sim c_f x\) for some constant \(c_f>0\). 
This transition in the asymptotic size of the variance when \(x\approx q^2\) is a (\(q\)-aspect) analogue of the transition we observe (in \(k\)-aspect) when \(x\approx k^2\). 
However, the second transition when \(x\approx k\) (over which we are able to asymptotically compute the variance) is a feature not present in the \(q\)-aspect problem (when \(x\leq q\), the sums in progressions contain at most one term).

Whilst Lau and Zhao establish a bound \(\mathrm{Var}\ll_f q x^{1/2}\) for \(q^{2-\epsilon}\leq x\leq q^{2+\epsilon}\), the exact asymptotic behaviour of the variance (\ref{blomervar}) is unclear over this transition.
Interestingly, Fouvry, Ganguly, Kowalski and Michel \cite{fgkm} (see also \cite{lesteryesha}) prove a central limit theorem for these sums in progressions (on average over the congruence class) in this transition regime, when \(q\) is prime and \(x=q^{2-o(1)}\). 

In another instance, Petrow \cite{petrow} has studied shifted convolution sums \(\sum_n \lambda_f(n)\lambda_f(n+h)\). 
These are related to the off-diagonal terms appearing in this article, since one has \((\sum_n\lambda_f(n))^2=\sum_n \lambda_f(n)^2+\sum_{h\neq0}\sum_n\lambda_f(n)\lambda_f(n+h)\). 
Petrow detects a transition in the first moment of these shifted convolution sums, taken on average over the shifts \(h\), when the length of the sums is \(\approx h^2\). 
See also the work of Conrey, Farmer and Soundararajan \cite{conreyfarmersound} for another example of such a transition in sums of the Jacobi symbol.

The proof of Theorems \ref{thmmean} and \ref{thmvarsimple} relies upon the Petersson trace formula to compute the \(f\)-averages. This splits the first and second moments into diagonal and off-diagonal terms. The off-diagonal terms contain the Bessel functions \(J_{k-1}\). Essentially, the off-diagonal is shown to be small, except at the transitions (where \(x\in [k^2/(32\pi^2), k^2/(16\pi^2)]\) for the first moment and \(x\in[k/(8\pi), k/(4\pi)]\) for the second moment). Here, large peaks of the Bessel function produce a large off-diagonal contribution. 

Our handling of these off-diagonal contributions is analogous to work of Hough \cite{hough}, who proves a second moment estimate for the \(L\)-functions attached to the eigenvalues \((\lambda_f(n))_{n\geq1}\) in weight aspect (see also \cite[Theorem 4.2]{balkanovafrolenkov}). After an application of the Petersson trace formula, Hough (using Poisson summation) similarly extracts a secondary main term from the off-diagonal contribution.

\begin{remark} 
This paper is organised as follows. In the following section we state some basic facts that will be required throughout the text, including the Petersson trace formula and a Vorono\"i type summation formula for the sums \(\mathcal S(x,f)\). The Bessel functions \(J_{k-1}\) arise in both of these, and so understanding their behaviour is key. Appendix \ref{besselappendix} contains an overview of the Bessel functions, including various facts and asymptotics that will be used throughout. Equipped with these preliminaries, we prove Theorem \ref{thmmean} in Section \ref{secmean} and Theorem \ref{thmvarsimple} in Section \ref{secvar}.
\end{remark}

\subsection*{Acknowledgements} I would like to thank Stephen Lester for many helpful discussions and useful comments on an earlier draft of this paper. I also thank Bingrong Huang for pointing out the relevant work \cite{hough}. This work was supported by the Additional Funding Programme for Mathematical Sciences, delivered by EPSRC (EP/V521917/1) and the Heilbronn Institute for Mathematical Research.

\section{Preliminaries}\label{secprelim}

\subsection{Notation}
We use the standard notation \(e(z)=e^{2\pi i z}\) throughout. 
The divisor function is denoted by \(d(n)=\sum_{d\mid n}1\).
We also denote by \(a^*\Mod c\) the multiplicative inverse of \(a \Mod c\) (in other words \(aa^*\equiv 1\Mod c\)).

We use big \(\mathcal O\) notation, writing \(F(x)=\mathcal O(G(x))\) if \(|F(x)|\leq C|G(x)|\) for some constant \(C>0\) and all valid inputs \(x\). 
We also write \(F(x)\ll G(x)\) to indicate \(F(x)=\mathcal O(G(x))\) (or \(F(x)\gg G(x)\) to indicate \(G(x)=\mathcal O(F(x))\)).
We write \(F(x)\asymp G(x)\) to indicate \(F(x)\ll G(x)\ll F(x)\).
If the implied constants depend on some parameters, say \(p_1,p_2\ldots\), this can be stated explicitly or indicated with subscripts: \(\mathcal O_{p_1,p_2\ldots}\) or \(\ll_{p_1,p_2\ldots}\). 
Throughout this article, \(\epsilon\) will denote an arbitrarily small positive quantity, and the implied constants in general depend on \(\epsilon\) without mention. Dependence on other parameters is made explicit throughout.

Finally, given \(\Delta\geq1\), throughout this paper we will write \(w=w_\Delta\) for a smooth function \(w:\mathbb{R}\to \mathbb{R}\) satisfying the following:
\begin{equation}\label{wdef}
\begin{gathered}
\supp w= [1,2],\\
w(\xi)=1 \text{ for } 1+\Delta^{-1}\leq \xi\leq 2-\Delta^{-1},\\
\text{for all integers } j\geq 0 \text{ and all } \xi, \text{ we have } w^{(j)}(\xi)\ll_j \Delta^j.
\end{gathered}
\end{equation}

\subsection{The Petersson Trace Formula}

The key tool for computing \(f\)-averages is the Petersson trace formula. Since the \(f\)-averages are normalised by the harmonic weights (see (\ref{favs})), this reads as follows (see \cite[Theorem 3.6]{iwaniec}).

\begin{lemma}[Petersson Trace Formula]\label{trace}
Let \(S(m,n;c)\) denote the Kloosterman sums 
\[S(m,n;c)=\sum_{\substack{a\Mod c\\ (a,c)=1}} e\Big(\frac{a^*m+an}{c}\Big),\]
Let \(\langle\cdot \rangle\) be given by (\ref{favs}). Then, for any positive integers \(m\) and \(n\), we have
\begin{equation}\label{pet}
\langle \lambda_f(n)\lambda_f(m)\rangle=\delta_{mn}+2\pi (-1)^{k/2}\sum_{c\geq1} c^{-1}S(m,n;c)J_{k-1}\Big(\frac{4\pi\sqrt{mn}}{c}\Big),
\end{equation}
where \(J_{k-1}\) denotes the Bessel function and 
\[\delta_{mn}=\begin{cases} 1 &\text{ if } m=n,\\ 0 &\text{ otherwise}.\end{cases}\]
\end{lemma}

When the integers \(m\) and \(n\) in (\ref{pet}) are small compared to the weight \(k\), the Bessel functions appearing in (\ref{pet}) are very small. Consequently, the trace formula effectively functions as an orthogonality relation in this case.

\begin{corollary}\label{averages}
Let \(k\) be sufficiently large, and \(\langle \cdot\rangle\) be as given in (\ref{favs}). Let \(m\) and \(n\) be positive integers satisfying \(4\pi \sqrt{mn}\leq 999k/1000\). Then
\[\langle \lambda_f(n)\lambda_f(m)\rangle =\delta_{mn}+\mathcal O(\exp(-k^{3/5})).\]
\end{corollary}

\begin{proof}
We show that the off-diagonal contribution in (\ref{pet}) is \(\mathcal O(\exp(-k^{3/5}))\). Trivially bounding the Kloosterman sums \(|S(m,n;c)|\leq c\), the off-diagonal is seen to be
\begin{equation}\label{offdiagpetersson}
2\pi (-1)^{k/2}\sum_{c\geq1} c^{-1}S(m,n;c)J_{k-1}\Big(\frac{4\pi\sqrt{mn}}{c}\Big)\ll \sum_{c\geq1} \Big|J_{k-1}\Big(\frac{4\pi\sqrt{mn}}{c}\Big)\Big|.
\end{equation}
If \(1\leq c<4\), then for large enough \(k\)
\[\frac{4\pi\sqrt{mn}}c \leq 4\pi\sqrt{mn} \leq \frac{999k}{1000}\leq k-k^{1/3+3/5}\implies J_{k-1}\Big(\frac{4\pi\sqrt{mn}}{c}\Big)\ll \exp(-k^{3/5}),\]
using the Bessel function bound (\ref{bessel1}) of Lemma \ref{besselbounds}. On the other hand, if \(c\geq4\) then 
\[\frac{4\pi \sqrt{mn}}{c}\leq \frac{999k}{4000}\leq \frac {k}4\implies J_{k-1}\Big(\frac{4\pi\sqrt{mn}}{c}\Big)\ll \frac{mn}{c^2}e^{-14k/13},\]
by the Bessel function bound (\ref{bessboundsmallarg}) of Lemma \ref{besselbounds}. Combining the above two bounds with (\ref{offdiagpetersson}) shows the off-diagonal is 
\[\ll \exp(-k^{3/5})+\sum_{c\geq4}\frac{mn}{c^2}e^{-14k/13}\ll \exp(-k^{3/5}),\]
as claimed.
\end{proof}

%\begin{corollary}\label{averages}
%Let \(m\) and \(n\) be positive integers satisfying \(\sqrt{mn}\leq k/(16\pi)\). Then 
%\[\langle \lambda_f(n)\lambda_f(m)\rangle=\delta_{mn}+\mathcal{O}(e^{-15k/14}).\]
%\end{corollary}

%\begin{proof}
%We must bound the off-diagonal contribution appearing in (\ref{pet}). Since we assume the argument of the Bessel functions is \(4\pi\sqrt{mn}/c\leq k/4\), we have the bound (\ref{bessboundsmallarg}) for the Bessel functions, which states \(J_{k-1}(4\pi\sqrt{mn}/c)\ll mne^{-14k/13}/c^2\). Trivially bounding the Kloosterman sums \(|S(m,n;c)|\leq c\), we have
%\[2\pi (-1)^{k/2}\sum_{c\geq1} c^{-1}S(m,n;c)J_{k-1}\Big(\frac{4\pi\sqrt{mn}}{c}\Big)\ll \sum_{c\geq 1}\frac{mn}{c^2} e^{-14k/13}\ll e^{-15k/14}.\]
%\end{proof}

\subsection{A Vorono\"i Summation Formula}

The second result we need is a standard Vorono\"i-type summation formula for sums of coefficients of cusp forms. In order to apply this to the sharp cut-off sums \(\mathcal S(x,f)\)
we first introduce a smoothing.

%\begin{definition}\label{wdef}
%Given \(\Delta\geq1\), denote by \(w=w_\Delta\) a smooth function \(w:\mathbb{R}\to \mathbb{R}\) satisfying the following:
%\begin{itemize}
%\item \(\supp w= [1,2],\)
%\item \(w(\xi)=1\) for \(1+\Delta^{-1}\leq \xi\leq 2-\Delta^{-1},\)
%\item for all integers \(j\geq0\) and all \(\xi\), we have \(w^{(j)}(\xi)\ll_j \Delta^j\).
%\end{itemize}
%\end{definition}

\begin{lemma}\label{vorprop}
Let \(\Delta\) be a sufficiently large parameter satisfying \(\Delta\leq x^{1-\epsilon}\) for some \(\epsilon>0\), and let \(w=w_\Delta\) be the associated smooth function given in (\ref{wdef}). Then for a Hecke eigenform \(f\) of weight \(k\) for \(\mathrm{SL}_2(\mathbb{Z})\), normalised so that \(\lambda_f(1)=1\), we have
\begin{equation*}
\mathcal S(x,f)=2\pi (-1)^{k/2} x\sum_{n\geq1} \lambda_f(n)\tilde w\Big(\frac{nx}{k^2+\Delta^2}\Big)+\mathcal O \Big(\frac{x\log x}{\Delta}\Big),
\end{equation*}
where 
\begin{equation}\label{tildew}
\tilde{w}(\xi)=\tilde w_\Delta(\xi)=\int_0^\infty w(t)J_{k-1}(4\pi \sqrt{(k^2+\Delta^2)\xi t})dt.
\end{equation}
\end{lemma}

\begin{proof}
We claim
\begin{equation}\label{smoothing}
\mathcal S(x,f)=\sum_{n\geq 1} \lambda_f(n) w\Big(\frac{n}{x}\Big)+\mathcal O \Big(\frac{x\log x}{\Delta}\Big).
\end{equation}
Indeed, by Deligne's bound (\(|\lambda_f(n)|\leq d(n)\)), the error in approximating \(\mathcal S(x,f)\) by the smoothed sums in (\ref{smoothing}) is
\begin{equation}\label{error1}
\ll \sum_{\substack{x\leq n\leq (1+\Delta^{-1})x\\
\text{or }(2-\Delta^{-1})x\leq n\leq 2x}}d(n).
\end{equation}
A result of Shiu \cite[Theorem 1]{shiu} (applicable since \(\Delta\leq x^{1-\epsilon}\)) shows (\ref{error1}) is
\[\ll \frac{x}{\Delta \log x} \exp\Big(\sum_{p\leq 2x}\frac{d(p)}{p}\Big)= \frac{x}{\Delta \log x} \exp\big(2\log\log{2x}+\mathcal O(1)\big)\ll \frac{x\log x}{\Delta}.\]
This proves (\ref{smoothing}). It is a standard matter to transform the smoothed sums -- the lemma now follows at once from (\ref{smoothing}) and the smooth Vorono\"i summation formula for coefficients of cusp forms found in \cite[ex.9, p.83]{iwanieckowalski}.
\end{proof}

The effective length of the transformed sums in Lemma \ref{vorprop} is around \((k^2+\Delta^2)/x\). Indeed, the weights \(\tilde w\) decay rapidly, as is established concretely in the following lemma.

\begin{lemma}\label{ibp}
For \(\xi>0\) and any integer \(A\geq 0\), we have
\[\tilde w(\xi)\ll_A \xi^{-A/2}.\]
\end{lemma}

\begin{proof}
This is a standard application of integration by parts (see for example \cite[Lemma 2.1]{lesteryesha}, \cite[Lemma 3.1]{lauzhao} or \cite[(9)]{blomer}) so we give only a brief sketch. 

It follows from (\ref{besseldiff1}) that for any \(\nu\geq 0\),
\begin{multline*} 
\frac{d}{dt}\Big\{(16\pi^2(k^2+\Delta^2)\xi t)^{\nu/2}J_\nu(4\pi\sqrt{(k^2+\Delta^2)\xi t})\Big\}\\
=\frac12 (16\pi^2(k^2+\Delta^2)\xi)^{(\nu+1)/2} t^{(\nu-1)/2} J_{\nu-1}(4\pi\sqrt{(k^2+\Delta^2)\xi t}).
\end{multline*}
Integrating by parts, we obtain 
\begin{multline}\label{ibpedfirsttime}
\tilde w(\xi)=\int_0^\infty w(t)J_{k-1}(4\pi\sqrt{(k^2+\Delta^2)\xi t})dt\\
=2(16\pi^2(k^2+\Delta^2)\xi )^{-(k+1)/2} \int_0^\infty w(t) t^{-(k-1)/2} d\big( (16\pi^2(k^2+\Delta^2)\xi t)^{k/2} J_{k}(4\pi\sqrt{(k^2+\Delta^2)\xi t})\big)\\
=\frac{k-1}{4\pi\sqrt{(k^2+\Delta^2)\xi}}\int_0^\infty t^{-1/2} w(t) J_k(4\pi\sqrt{(k^2+\Delta^2)\xi t})dt\\
-\frac1{2\pi\sqrt{(k^2+\Delta^2)\xi}}\int_0^\infty t^{1/2}w'(t)J_k(4\pi \sqrt{(k^2+\Delta^2)\xi t})dt.
\end{multline}
Since \(\supp w=[1,2]\), we apply the bounds \(|J_n(z)|\leq 1\) (valid for integer \(n\) and \(z\geq0\), see (\ref{besselintrep2})) and \(w^{(j)}(t)\ll_j \Delta^j\) (valid for any integer \(j\geq0\)) to conclude
\[\tilde w(\xi)\ll \frac{k+\Delta}{\sqrt{(k^2+\Delta^2)\xi}}\ll \xi^{-1/2}.\]
This proves the case \(A=1\). Integrating (\ref{ibpedfirsttime}) by parts a further \(A-1\) times, similarly we deduce the bound
\begin{multline*}
\tilde w(\xi)\ll \big((k^2+\Delta^2)\xi\big)^{-A/2}\sum_{0\leq j\leq A} k^{A-j}\int_0^\infty t^{-A/2+j}|w^{(j)}(t)J_{k-1+A}(4\pi \sqrt{(k^2+\Delta^2)\xi t})|dt\\
\ll_A \big((k^2+\Delta^2)\xi\big)^{-A/2}\sum_{0\leq j\leq A} k^{A-j}\Delta^j\ll_A \xi^{-A/2}.
\end{multline*}
\end{proof}

\subsection{Oscillatory Integrals}

Finally, we record a lemma which will be used to bound oscillatory integrals. This is essentially a very general formulation of integration by parts, due to Blomer, Khan and Young \cite[Lemma 8.1]{stationaryphase}.

\begin{lemma}[Integration by Parts]\label{bkyibp}
Let \(Q,U,R,X>0\) and \(Y\geq 1\) be some parameters. Let \(\rho\) and \(\phi\) be two smooth functions. Assume \(\rho\) is compactly supported on the interval \([\alpha,\beta]\) and satisfies
\[\rho^{(j)}(t)\ll_j XU^{-j} \text{ for } j=0,1,2\ldots\]
Assume \(\phi\) satisfies
\[\phi'(t)\gg R \text{ and } \phi^{(j)}(t)\ll_j YQ^{-j} \text{ for } j=2,3,\ldots\]
Then for any integer \(B\geq 0\),
\begin{equation}\label{intbound}
\int_\alpha^\beta \rho(t)e^{i\phi(t)}dt\ll_B (\beta-\alpha)X\Big\{\Big(\frac{QR}{Y^{1/2}}\Big)^{-B}+(RU)^{-B}\Big\}.
\end{equation}
\end{lemma}

\section{The First Moment: Proof of Theorem \ref{thmmean}}\label{secmean}

The first part of Theorem \ref{thmmean} is an essentially trivial consequence of the Petersson trace formula.

\begin{proof}[Proof of Theorem \ref{thmmean}, part \ref{meani}]
Our assumption \(x\leq k^2/(32\pi^2+1)\) implies \(4\pi\sqrt{2x}\leq k(1+1/(32\pi^2))^{-1/2}\leq 999k/1000\). Consequently, we can apply the Petersson trace formula in the form of Corollary \ref{averages}. This shows
\begin{equation*} 
\langle \mathcal S(x,f)\rangle =\sum_{x\leq n\leq 2x}\langle \lambda_f(n)\lambda_f(1)\rangle \ll x\exp(-k^{3/5}) \ll e^{-\sqrt k}.
\end{equation*}
\end{proof}

The proof of part \ref{meanii} is rather more involved. In the above, we had the strong bound \(\mathcal O(\exp(-k^{3/5}))\) (from Corollary \ref{averages}) for the off-diagonal contribution from the Petersson trace formula. The off-diagonal contribution to \(\langle \lambda_f(n)\lambda_f(1)\rangle\) is
\[2\pi(-1)^{k/2}\sum_{c\geq1} c^{-1}S(n,1;c)J_{k-1}\Big(\frac{4\pi\sqrt{n}}{c}\Big).\]
Unfortunately, the above is no longer negligible for \(x\leq n\leq 2x\) once \(x\geq k^2/(32\pi^2)\). This is because some of the Bessel functions \(J_{k-1}(4\pi\sqrt{n}/c)\) are now close to their peak (if \(x\geq k^2/(32\pi^2)\) then \(4\pi\sqrt{n}/c\approx k\) for some \(x\leq n\leq 2x\) and \(c\geq1\)). This produces the large contribution appearing in part \ref{meanii} of Theorem \ref{thmmean}. For this reason, it is convenient to apply the Vorono\"i summation formula given in Lemma \ref{vorprop}, which reveals this off-diagonal contribution. This is done in the following lemma.

\begin{lemma}\label{vorprop2}
Suppose \(k^2/(64\pi^2)\leq x\leq k^4\). Then for any \(\epsilon>0\), one has
\begin{equation*}
\langle \mathcal S(x,f)\rangle =\frac{(-1)^{k/2}}{4\pi}\int_{4\pi\sqrt x}^{4\pi \sqrt{2x}}yJ_{k-1}(y)dy+\mathcal O(x^{1/2}k^{-1+\epsilon}).
\end{equation*}
\end{lemma}

\begin{proof}
Throughout this proof, fix \(\epsilon>0\) and set \(\Delta=x^{1/2}k^{1-\epsilon/2}\). Note (by our assumptions on \(x\)) that \(\Delta\leq 8\pi xk^{-\epsilon/2}\leq x^{1-\epsilon/10}\), say. Thus we may apply Lemma \ref{vorprop} with this choice of \(\Delta\) to see that for \(f\in\mathcal B_k\), 
\begin{equation}\label{transtruncc}
\mathcal S(x,f)=2\pi (-1)^{k/2}x\sum_{n\geq 1} \lambda_f(n)\tilde w\Big(\frac{nx}{k^2+\Delta^2}\Big)+\mathcal O \Big(\frac{x^{1/2}\log x}{k^{1-\epsilon/2}}\Big),
\end{equation}
where
\begin{equation*}
\tilde{w}(\xi)=\tilde{w}_\Delta(\xi)=\int_0^\infty w_\Delta(t)J_{k-1}(4\pi \sqrt{(k^2+\Delta^2)\xi t})dt.
\end{equation*}
Using Lemma \ref{ibp}, we may truncate the sum (\ref{transtruncc}). Indeed, if \(n\geq(k^2+\Delta^2)k^{\epsilon/2}/x\), then Lemma \ref{ibp} shows
\[\tilde w\Big(\frac{nx}{k^2+\Delta^2}\Big) \ll n^{-2} k^{-1000},\]
say. Noting also that (since \(x\geq k^2/(64\pi^2)\)) we have \((k^2+\Delta^2)k^{\epsilon/2}/x=x^{-1}k^{2+\epsilon/2}+k^{2-\epsilon/2}\leq 2k^{2-\epsilon/2},\)
we deduce from (\ref{transtruncc}) that
\begin{equation}\label{transtruncc2}
\mathcal S(x,f)=2\pi (-1)^{k/2}x\sum_{n\leq 2k^{2-\epsilon/2}} \lambda_f(n)\tilde w\Big(\frac{nx}{k^2+\Delta^2}\Big)+\mathcal O (x^{1/2}k^{-1+\epsilon}).
\end{equation}
We now apply the Petersson trace formula in the form of Corollary \ref{averages}, which yields
\begin{equation}\label{petapplied201}
\langle \mathcal S(x,f)\rangle=2\pi (-1)^{k/2} x\tilde w\Big(\frac{x}{k^2+\Delta^2}\Big)+ \mathcal O (x^{1/2}k^{-1+\epsilon}).
\end{equation}
(The off-diagonal terms contribute \(\ll e^{-\sqrt k}\), as in the proof of part \ref{meani}.) Compute
\begin{multline}\label{bessint}
\tilde w\Big(\frac{x}{k^2+\Delta^2}\Big)
=\int_0^\infty w(t)J_{k-1}(4\pi\sqrt{xt})dt\\
=\int_1^2J_{k-1}(4\pi\sqrt{xt})dt+\mathcal O\Big(\Big\{\int_1^{1+\Delta^{-1}}+\int_{2-\Delta^{-1}}^2\Big\}|J_{k-1}(4\pi\sqrt{xt})|dt\Big)\\
=\frac{1}{8\pi^2x}\int_{4\pi\sqrt{x}}^{4\pi\sqrt{2x}} yJ_{k-1}(y)dy+\mathcal O(k^{-1/3}\Delta^{-1}).
\end{multline}
In the last line, we substituted \(y=4\pi\sqrt{xt}\) and used the bound \(J_{k-1}(z)\ll k^{-1/3}\) given in part \ref{besiii} of Lemma \ref{besselbounds}. Combining (\ref{petapplied201}) and (\ref{bessint}) now proves the lemma.
\end{proof}

In the range of part \ref{meani} of Theorem \ref{thmmean}, the Bessel function appearing in Lemma \ref{vorprop2} is negligible. In the range considered in part \ref{meanii}, the transition regime and part of the oscillatory regime of the Bessel function are contained in the region of integration (which is \([4\pi\sqrt x, 4\pi\sqrt{2x}]\)). The following two lemmas estimate \(\int yJ_{k-1}(y)dy\) in these two different regimes of the Bessel function. We first compute the contribution of the transition regime. 

\begin{lemma}\label{transrangeint}
Let \(\nu\) be sufficiently large. Then
\[\int_{(\nu+1)-(\nu+1)^{1/2}}^{(\nu+1)+(\nu+1)^{1/2}}yJ_{\nu}(y)dy=\nu+\mathcal O(\nu^{7/8}).\]
\end{lemma}

\begin{proof}
First note that the bound (\ref{bessel2}) of Lemma \ref{besselbounds} shows 
\begin{align}\label{lem321} 
\int_{(\nu+1)-(\nu+1)^{1/2}}^{(\nu+1)+(\nu+1)^{1/2}}&yJ_\nu(y)dy=\int_{\nu-\nu^{1/2}}^{\nu+\nu^{1/2}}yJ_{\nu}(y)dy\\
&\hspace{4em}+\mathcal O\Big(\Big\{\int_{\nu+\nu^{1/2}}^{(\nu+1)+(\nu+1)^{1/2}}+\int_{\nu-\nu^{1/2}}^{(\nu+1)-(\nu+1)^{1/2}}\Big\}|yJ_{\nu}(y)|dy\Big)\nonumber\\
&=\nu\int_{\nu-\nu^{1/2}}^{\nu+\nu^{1/2}}  J_{\nu}(y)dy+\mathcal O\Big(\int_{\nu-\nu^{1/2}}^{\nu+\nu^{1/2}}|\nu-y||J_\nu(y)|dy\Big)+\mathcal O(\nu^{2/3})\nonumber\\
&=\nu^{4/3}\int_{-\nu^{1/6}}^{\nu^{1/6}}J_\nu(\nu+\nu^{1/3}z)dz+\mathcal O(\nu^{2/3}).\nonumber
\end{align}
We now apply the asymptotic (\ref{krasikovgood}) of Lemma \ref{besselasymptotics}, which shows 
\begin{multline}\label{lem322}
\nu^{4/3}\int_{-\nu^{1/6}}^{\nu^{1/6}}J_\nu(\nu+\nu^{1/3}z)dz=2^{1/3}\nu\int_{-\nu^{1/6}}^{\nu^{1/6}}\Ai(-2^{1/3}z)dz+\mathcal O\Big(\nu^{1/3}\int_{-\nu^{1/6}}^{\nu^{1/6}}(z^{9/4}+1)dz\Big)\\
=\nu\int_{-2^{1/3}\nu^{1/6}}^{2^{1/3}\nu^{1/6}}\Ai(z)dz+\mathcal O(\nu^{7/8}).
\end{multline}
The remaining Airy integral can be evaluated by extending the range of integration to infinity. Using the asymptotic (\ref{airyneg}) for the Airy function, one has for any sufficiently large parameter \(\xi\)
\begin{align*}
\int_{-\infty}^{-\xi}\Ai(z)dz&=\frac{1}{\sqrt \pi}\int_\xi^\infty z^{-1/4}\cos\Big(\frac{2}{3}z^{3/2}-\frac{\pi}{4}\Big)dz+\mathcal O\Big(\int_{\xi}^\infty z^{-7/4}dz\Big)\\
&=\frac{1}{\sqrt\pi}\int_{\xi}^\infty z^{-3/4}d\Big(\sin\Big(\frac{2}{3}z^{3/2}-\frac{\pi}{4}\Big)\Big)+\mathcal O(\xi^{-3/4})\\
&=-\frac{1}{\sqrt\pi}\xi^{-3/4}\sin\Big(\frac{2}{3}\xi^{3/2}-\frac{\pi}{4}\Big)+\frac{3}{4\sqrt\pi}\int_{\xi}^\infty z^{-7/4}\sin\Big(\frac{2}{3}z^{3/2}-\frac{\pi}{4}\Big)dz+\mathcal O(\xi^{-3/4})\\
&\ll \xi^{-3/4}.
\end{align*}
Similarly (using (\ref{airypos})) one can show
\[\int_\xi^\infty \Ai(z)dz\ll \xi^{-3/4}.\]
Consequently, using (\ref{airyint}) we have
\begin{equation}\label{lem323} 
\nu\int_{-2^{1/3}\nu^{1/6}}^{2^{1/3}\nu^{1/6}}\Ai(z)dz=\nu\int_{-\infty}^\infty \Ai(z)dz+\mathcal O(\nu^{7/8})=\nu+\mathcal O(\nu^{7/8}).
\end{equation}
Combining (\ref{lem321}), (\ref{lem322}) and (\ref{lem323}) completes the proof. 
\end{proof}

Next, we bound the contribution of the oscillatory regime.

\begin{lemma}\label{oscillrangeint}
Let \(\nu\) be sufficiently large and \(\nu+\nu^{1/3+\epsilon}\leq \alpha\leq \beta\leq 2\alpha\). Then
\begin{equation*}
\int_{\alpha}^\beta yJ_{\nu}(y)dy\ll \alpha^2 (\alpha^2-\nu^2)^{-3/4}.
\end{equation*}
\end{lemma}

\begin{proof}
Lemma \ref{besselasymptotics} gives an asymptotic expansion of \(J_\nu(y)\) valid for all \(y\in[\alpha,\beta]\). Indeed, since \(\alpha\geq \nu+\nu^{1/3+\epsilon}\), (\ref{bessel3trunc}) shows
\begin{multline}\label{bessasympreplaced}
\int_\alpha^\beta yJ_\nu(y)dy=\sqrt{\frac2\pi}\int_\alpha^\beta y(y^2-\nu^2)^{-1/4}\cos\omega(y)dy+\mathcal O\Big(\int_\alpha^\beta y^3(y^2-\nu^2)^{-7/4}dy\Big)\\
=\frac{1}{\sqrt{2\pi}}\int_\alpha^\beta y (y^2-\nu^2)^{-1/4}(e^{i\omega(y)}+e^{-i\omega(y)})dy +\mathcal O(\alpha^2(\alpha^2-\nu^2)^{-3/4}).
\end{multline}
It is a standard matter to bound the oscillatory integrals above. Recall (\ref{omega'}), which states \(\omega'(y)=(y^2-\nu^2)^{1/2}/y\). Applying \cite[Lemma 4.3]{titchmarsh} (possibly after partitioning the interval of integration into two subintervals on which \(y^2(y^2-\nu^2)^{-3/4}\) is monotonic) produces the bound
\[\int_\alpha^\beta y(y^2-\nu^2)^{-1/4}e^{\pm i\omega(y)} dy\ll \alpha^2(\alpha^2-\nu^2)^{-3/4}.\]
The lemma follows immediately upon replacing this estimate in (\ref{bessasympreplaced}).
\end{proof}

\begin{proof}[Proof of Theorem \ref{thmmean}, part \ref{meanii}]
Our assumption \(k^2/(32\pi^2-1)\leq x\leq k^2/(16\pi^2+1)\) implies \([k-k^{1/2},k+k^{1/2}]\subset [4\pi\sqrt{x},4\pi\sqrt{2x}]\). We thus deduce from Lemma \ref{vorprop2} that
\begin{equation}\label{splitting402}
\langle \mathcal S(x,f)\rangle =\frac{(-1)^{k/2}}{4\pi}\Big\{\int_{4\pi\sqrt{x}}^{k-k^{1/2}}+\int_{k-k^{1/2}}^{k+k^{1/2}}+\int_{k+k^{1/2}}^{4\pi\sqrt{2x}}\Big\}yJ_{k-1}(y)dy+O(x^{1/2}k^{-1+\epsilon}).
\end{equation}
The first of these integrals is easily bounded using (\ref{bessel1}) of Lemma \ref{besselbounds}, which shows
\[y\leq k-k^{1/2}\implies J_{k-1}(y)\ll e^{-k^{1/6}}\implies \int_{4\pi\sqrt{x}}^{k-k^{1/2}}yJ_{k-1}(y)dy\ll e^{-k^{1/6}/2}.\]
So (\ref{splitting402}) shows
\begin{equation*}
\langle \mathcal S(x,f)\rangle =\frac{(-1)^{k/2}}{4\pi}\int_{k-k^{1/2}}^{k+k^{1/2}}yJ_{k-1}(y)dy+\frac{(-1)^{k/2}}{4\pi}\int_{k+k^{1/2}}^{4\pi\sqrt{2x}}yJ_{k-1}(y)dy+O(x^{1/2}k^{-1+\epsilon}).
\end{equation*}
The former integral is equal to \(k+\mathcal O(k^{7/8})\) by Lemma \ref{transrangeint}. By Lemma \ref{oscillrangeint}, the latter integral is
\begin{equation*}
\int_{k+k^{1/2}}^{4\pi\sqrt{2x}}yJ_{k-1}(y)dy\ll k^2((k+k^{1/2})^2-k^2)^{-3/4}\ll k^{7/8}.
\end{equation*}
(Note our assumption \(x\leq k^2/(16\pi^2+1)\) ensures \(4\pi\sqrt{2x}\leq 2(k+k^{1/2})\)).
\end{proof}

\section{The Second Moment: Proof of Theorem \ref{thmvarsimple}}\label{secvar}

Our next goal is to evaluate the second moment of the sums. This is similar but more involved than the calculations for the first moment in the previous section. We will prove the following, more precise version of Theorem \ref{thmvarsimple}.

\begin{theorem}\label{thmvar}
Let \(\epsilon>0\). We have the following asymptotics for the second moment of the sums \(\mathcal S(x,f)\).
\begin{enumerate}[label=(\roman*)]
    \item \label{vari} If \(x\leq k/(32\pi)\), then 
    \[\langle \mathcal S(x,f)^2\rangle =x+\mathcal O(1).\]
    \item \label{varii}
    If \(k/(32\pi)\leq x\leq k^{1+\epsilon}\), then
    \begin{equation*}
    \langle \mathcal S(x,f)^2\rangle=x+\frac{(-1)^{k/2}}{2\pi}L\Big(\frac{k}{4\pi x}\Big)k(1+\mathcal O(k^{-1/8}))+\mathcal O_\epsilon(k^{2/3+4\epsilon}),
    \end{equation*}
    where \(L\) is the function defined in (\ref{ldef}). 
    \item \label{variii} If \(k^{1+\epsilon}\leq x\leq k^2\), then\footnote{This result only provides an asymptotic if \(x=o( k^2/\log^6k)\).}
    \[\langle \mathcal S(x,f)^2\rangle =x+\mathcal O_\epsilon\Big(\frac{x^{3/2}\log^3 k}{k}\Big).\]
\end{enumerate}
\end{theorem}

The basic idea is to expand the sums 
\begin{equation}\label{expand}
\mathcal S(x,f)^2 = \sum_{x<m,n\leq 2x} \lambda_f(n)\lambda_f(m),
\end{equation}
and then handle the resulting averages \(\langle \lambda_f(n)\lambda_f(m)\rangle\) using the Petersson trace formula. This is trivial when \(x\leq k/(32\pi)\):
\begin{proof}[Proof of Theorem \ref{thmvar}, part \ref{vari}]
Given \(x\leq k/(32\pi)\), one has \(4\pi\sqrt{mn}\leq k/4\) for all integers \(n,m\leq 2x\). Consequently, one may apply the Petersson trace formula in the form of Corollary \ref{averages} to see 
\[\langle \mathcal S(x,f)^2\rangle=\sum_{x<m,n\leq 2x}\langle \lambda_f(n)\lambda_f(m)\rangle = \sum_{x<m,n\leq 2x}\delta_{mn}+\mathcal O(x^2\exp(-k^{3/5}))=x+\mathcal O(1).\]
\end{proof}

The off-diagonal contribution to the second moment is
\begin{equation}\label{offdiagexplainer}
2\pi (-1)^{k/2} \sum_{x<m,n\leq 2x}\sum_{c\geq1} c^{-1}S(m,n;c)J_{k-1}\Big(\frac{4\pi\sqrt{mn}}{c}\Big).
\end{equation}
Roughly speaking, the only terms contributing to the sum (\ref{offdiagexplainer}) are those with \(c\leq 8\pi x/k\), since the Bessel functions \(J_{k-1}(4\pi \sqrt{mn}/c)\) are negligibly small if \(c> 8\pi x/k\).
(Above, it was proved that the off-diagonal contributes a negligible amount to the second moment under the condition \(x\leq k/(32\pi)\).)
When \(k/(8\pi)\leq x\leq k/(4\pi)\), we expect the off-diagonal contribution to be large, since in this range of \(x\) only the \(c=1\) term contributes to (\ref{offdiagexplainer}), and this term can be large.
This is because one has \(4\pi \sqrt{mn}\approx k\) for some \(x<m, n\leq 2x\). 
For these \(n, m\), the Bessel function \(J_{k-1}(4\pi \sqrt{mn})\) is close to its peak, and since also \(S(m,n;1)=1\) it is reasonable to expect the \(c=1\) term to be large. 
However, once \(x\geq k/(4\pi)\), many different terms in the \(c\)-sum contribute to (\ref{offdiagexplainer}). 
In this range, many of the Bessel functions \(J_{k-1}(4\pi\sqrt{mn}/c)\) are in their oscillatory regime, and we also expect cancellation from the Kloosterman sums in the \(c>1\) terms.
Consequently, when \(x>k/(4\pi)\) we expect the off-diagonal to be relatively small. 

This heuristic explains parts \ref{varii} and \ref{variii} of Theorem \ref{thmvar}.
The idea behind the proof of part \ref{varii} is to apply Poisson summation to the sums over \(n\) and \(m\) in (\ref{offdiagexplainer}), and treat the resulting integrals of Bessel functions in the same way as in Section \ref{secmean}. This will be done in Section \ref{thmonepart2}.
In the range of part \ref{variii}, where \(x\geq k^{1+\epsilon}\), the errors introduced by this method quickly become overwhelming. 
Part \ref{variii} is proved in Section \ref{subsecpart3}. 
The key is to apply the Vorono\"i type summation formula given in Lemma \ref{vorprop} before the Petersson trace formula. 
This transforms the sums (\ref{expand}) to have effective length \(k^2/x\), which is now considerably smaller than \(x\).

\subsection{Proof of Theorem \ref{thmvar}, part \ref{varii}}\label{thmonepart2}

In the range \(k/(32\pi)\leq x\leq k^{1+\epsilon}\), we compute the second moment using the following simple application of the Petersson trace formula. With a view to applying the Poisson summation formula later on, it is first convenient to smooth the sums.

\begin{lemma}\label{insplit4lem}
Let \(0<\epsilon\leq 1/100\), and assume \(k/(32\pi)\leq x\leq k^{1+\epsilon}\). Set \(\Delta=k^{1-\epsilon}\), and let \(w=w_\Delta\) be the associated smooth function given in (\ref{wdef}). Then 
\begin{equation*}
\langle \mathcal S(x,f)^2\rangle =x+(\mathrm{OD})+\mathcal O((x+(\mathrm{OD}))^{1/2}k^{3\epsilon})+\mathcal O(k^{6\epsilon}),
\end{equation*}
where \((\mathrm{OD})\) denotes the off-diagonal contribution given by
\begin{equation}\label{offdiag100}
(\mathrm{OD})=2\pi (-1)^{k/2}\sum_{n,m\geq1}w\Big(\frac{n}{x}\Big)w\Big(\frac{m}{x}\Big) \sum_{c\geq 1}c^{-1} S(m,n;c) J_{k-1}\Big(\frac{4\pi\sqrt{mn}}{c}\Big).
\end{equation}
\end{lemma}

\begin{remark}
We fix \(0<\epsilon\leq 1/100\), assume \(k/(32\pi)\leq x\leq k^{1+\epsilon}\), and take \(\Delta=k^{1-\epsilon}\) and \(w=w_\Delta\) (as above) throughout the rest of this section (§\ref{thmonepart2}). All the following lemmas (in §\ref{thmonepart2}) hold under these implicit assumptions.
\end{remark}

\begin{proof}
Smoothing the sums with the function \(w=w_\Delta\) (see (\ref{smoothing})), for \(f\in\mathcal B_k\) we obtain
\[\mathcal S(x,f)=\sum_{n\geq1}\lambda_f(n)w\Big(\frac{n}{x}\Big)+\mathcal O\Big(\frac{x\log x}{\Delta}\Big)=\sum_{n\geq1}\lambda_f(n)w\Big(\frac{n}{x}\Big)+\mathcal O(k^{3\epsilon}).\]
(Since \(x\leq k^{1+\epsilon}\).) We now compute
\begin{equation}\label{initialsplitting4v0}
\langle \mathcal S(x,f)^2\rangle =\sum_{n,m\geq1}\langle \lambda_f(n)\lambda_f(m)\rangle w\Big(\frac{n}{x}\Big)w\Big(\frac{m}{x}\Big) 
+\mathcal O\Big(\Big\langle \Big|\sum_{n\geq1}\lambda_f(n)w\Big(\frac{n}{x}\Big)\Big|\cdot k^{3\epsilon}\Big\rangle\Big)+\mathcal O(k^{6\epsilon}).
\end{equation}
The Petersson trace formula (Lemma \ref{trace}) shows
\begin{equation*}
\sum_{n,m\geq1}\langle \lambda_f(n)\lambda_f(m)\rangle w\Big(\frac{n}{x}\Big)w\Big(\frac{m}{x}\Big)=\sum_{n\geq1}w\Big(\frac{n}{x}\Big)^2+(\mathrm{OD}).
\end{equation*}
Moreover, (from the properties of \(w\) given in (\ref{wdef})) one easily has
\[\sum_{n\geq1}w\Big(\frac nx\Big)^2=x+\mathcal O\Big(\frac x\Delta\Big)=x+\mathcal O(k^{2\epsilon}).\]
The lemma now follows from (\ref{initialsplitting4v0}), upon applying the Cauchy-Schwarz inequality to bound the first error term appearing there.
\end{proof} 

The remainder of this section is devoted to evaluating the off-diagonal terms \((\mathrm{OD})\). The key idea is to interchange the order of summation, and then apply Poisson summation to the \(n\) and \(m\) sums in (\ref{offdiag100}). It turns out that the only significant contribution comes from the zero frequency associated to the \(c=1\) term on the dual side, which contributes roughly 
\[\int_x^{2x}\int_x^{2x}J_{k-1}(4\pi\sqrt{uv})dudv.\]
Since the Bessel function \(J_{k-1}(z)\) is negligible when \(z\leq k-k^{1/3+\epsilon}\), and oscillatory when \(z\geq k+k^{1/3+\epsilon}\) (which leads to a lot of cancellation), the above integral is small unless the transition regime of the Bessel function (which is roughly \([k-k^{1/3},k+k^{1/3}]\)) coincides with the region of integration. Since the argument of the Bessel function is \(4\pi\sqrt{uv}\in[4\pi x, 8\pi x]\), the secondary main term in part \ref{varii} of Theorem \ref{thmvar} appears only if \(k/(8\pi)\leq x\leq k/(4\pi)\).

This strategy is very similar to work of Hough \cite{hough}, in which an analogous application of Poisson summation is used to prove a second moment estimate for the \(L\)-functions \(L(s,f)=\sum_{n\geq1}\lambda_f(n)n^{-s}\).

\subsubsection{Poisson Summation} 

An application of Poisson summation yields the following. 
\begin{lemma} \label{initialpoisson}
The off-diagonal terms (\ref{offdiag100}) are given by 
\begin{multline*}
(\mathrm{OD})=2\pi (-1)^{k/2}\sum_{c\leq 32\pi x/k}c^{-1}\sum_{(\tilde n,\tilde m)\in\Gamma_c}\\
\int_0^\infty\int_0^\infty w\Big(\frac{u}{x}\Big)w\Big(\frac{v}{x}\Big)J_{k-1}\Big(\frac{4\pi\sqrt{uv}}{c}\Big)e\Big(\frac{\tilde nu+\tilde mv}{c}\Big)dudv
+\mathcal O(k^{-1000}).
\end{multline*}
The indexing sets \(\Gamma_c\) are given by:
\begin{itemize}
\item \(\Gamma_1=\{(-1,-1), (-1,0), (-1,1), (0,-1), (0,0), (0,1), (1,-1), (1,0), (1,1)\}.\)
\item \(\Gamma_2=\{(-1,-1), (-1,1), (1,-1), (1,1)\}.\)
\item For \(c\geq3\), \(\Gamma_c=\{(-1,-1), (1,1)\}\).
\end{itemize}
\end{lemma}

\begin{proof}
We begin with the expression (\ref{offdiag100}) for \((\mathrm{OD})\). First, we must truncate the summation over \(c\). Note that if \(c\geq 32\pi x/k\), then we have \(4\pi\sqrt{mn}/c\leq k/4\) for \(m,n\leq 2x\) in the range of summation in (\ref{offdiag100}). The bound (\ref{bessboundsmallarg}) of Lemma \ref{besselbounds} now shows \(J_{k-1}(4\pi\sqrt{mn}/c)\ll mn e^{-14k/13}/c^2\). So (using the trivial bound \(|S(m,n;c)|\leq c\) for the Kloosterman sums) the contribution of \(c\geq 32\pi x/k\) to \((\mathrm{OD})\) is seen to be \(\mathcal O(e^{-k})\). We thus obtain from (\ref{offdiag100})
\begin{multline}\label{offdiag101}
(\mathrm{OD})=2\pi (-1)^{k/2}\sum_{n,m\geq1}w\Big(\frac{n}{x}\Big)w\Big(\frac{m}{x}\Big) \sum_{c\leq 32\pi x/k}c^{-1} S(m,n;c) J_{k-1}\Big(\frac{4\pi\sqrt{mn}}{c}\Big)+\mathcal O(e^{-k})\\
=2\pi (-1)^{k/2}\sum_{c\leq 32\pi x/k} c^{-1}\sum_{\substack{1\leq a\leq c\\(a,c)=1}}\sum_{n,m\geq1} w\Big(\frac{n}{x}\Big)w\Big(\frac{m}{x}\Big)J_{k-1}\Big(\frac{4\pi\sqrt{mn}}{c}\Big)e\Big(\frac{an+a^*m}{c}\Big)+\mathcal O(e^{-k}).
\end{multline}
(In the last line, we opened the Kloosterman sums and interchanged the order of summation.) 
Now splitting the inner sums over \(n,m\) in (\ref{offdiag101}) into congruence classes modulo \(c\), we obtain
\begin{multline}\label{offdiag101v2}
(\mathrm{OD})=2\pi (-1)^{k/2}\sum_{c\leq 32\pi x/k} c^{-1}\sum_{\substack{1\leq a\leq c\\(a,c)=1}}\sum_{\gamma,\delta\Mod c}e\Big(\frac{a\gamma+a^*\delta}{c}\Big)\\
\hspace{4em}\sum_{n',m'\in\mathbb{Z}}w\Big(\frac{\gamma+n'c}{x}\Big)w\Big(\frac{\delta+m'c}{x}\Big)J_{k-1}\Big(\frac{4\pi\sqrt{(\gamma+n'c)(\delta+m'c)}}{c}\Big)+\mathcal O(e^{-k}).
\end{multline}
Poisson summation shows that the inner sum over \(n'\) and \(m'\) is equal to
\begin{multline*}
\sum_{\tilde n,\tilde m\in\mathbb{Z}}\int_{-\infty}^\infty \int_{-\infty}^\infty w\Big(\frac{\gamma+\tilde u c}{x}\Big)w\Big(\frac{\delta+\tilde v c}{x}\Big)J_{k-1}\Big(\frac{4\pi \sqrt{(\gamma+\tilde uc)(\delta+\tilde vc)}}{c}\Big) e(-\tilde n\tilde u-\tilde m\tilde v)d\tilde ud\tilde v\\
=\frac{1}{c^2}\sum_{\tilde n,\tilde m\in\mathbb{Z}}e\Big(\frac{\tilde n\gamma+\tilde m\delta}{c}\Big)\int_0^\infty\int_0^\infty w\Big(\frac{u}{x}\Big)w\Big(\frac{v}{x}\Big)J_{k-1}\Big(\frac{4\pi \sqrt{uv}}{c}\Big)e\Big(-\frac{\tilde nu+\tilde mv}{c}\Big)dudv.
\end{multline*}
In last line of the above, we made the substitutions \(u=\gamma+\tilde u c\) and \(v=\delta+\tilde v c\). Recall \(\supp w=[1,2]\), and therefore the effective region of integration is \(x\leq u,v\leq 2x\). Replacing this in (\ref{offdiag101v2}), we obtain
\begin{align}\label{offdiag101v3}
(\mathrm{OD})&=2\pi (-1)^{k/2}\sum_{c\leq 32\pi x/k} c^{-3}\sum_{\substack{1\leq a\leq c\\(a,c)=1}}\sum_{\tilde n,\tilde m \in\mathbb{Z}}\Big\{\sum_{\gamma\Mod c}e\Big(\frac{\gamma(a+\tilde n)}{c}\Big)\Big\}\Big\{\sum_{\delta\Mod c}e\Big(\frac{\delta(a^*+\tilde m)}{c}\Big)\Big\}\\
&\hspace{7em}\int_0^\infty\int_0^\infty w\Big(\frac{u}{x}\Big)w\Big(\frac{v}{x}\Big)J_{k-1}\Big(\frac{4\pi \sqrt{uv}}{c}\Big)e\Big(-\frac{\tilde nu+\tilde mv}{c}\Big)dudv+\mathcal O(e^{-k})\nonumber\\
&=2\pi (-1)^{k/2}\sum_{c\leq 32\pi x/k} c^{-1}\sum_{\substack{1\leq a\leq c\\(a,c)=1}}\sum_{\substack{\tilde n,\tilde m\in\mathbb{Z}\\\tilde n\equiv -a\Mod c\\\tilde m\equiv -a^*\Mod c}}\nonumber\\
&\hspace{7em}\int_0^\infty \int_0^\infty w\Big(\frac{u}{x}\Big)w\Big(\frac{v}{x}\Big)J_{k-1}\Big(\frac{4\pi\sqrt{uv}}{c}\Big)e\Big(-\frac{\tilde nu+\tilde mv}{c}\Big)dudv+\mathcal O(e^{-k})\nonumber.
\end{align}

It remains to truncate the dual sums over \(\tilde n,\tilde m\) in (\ref{offdiag101v3}). To do so, we use the integral representation (\ref{besselintrep2}) for the Bessel function. This shows that for any integers \(\tilde n\) and \(\tilde m\),
\begin{multline*}
\int_0^\infty\int_0^\infty w\Big(\frac{u}{x}\Big)w\Big(\frac{v}{x}\Big)J_{k-1}\Big(\frac{4\pi\sqrt{uv}}{c}\Big)e\Big(-\frac{\tilde nu+\tilde mv}{c}\Big)dudv\\
=\frac{1}{2\pi}\int_{-\pi}^\pi e^{i(k-1)\theta}\int_0^\infty\int_0^\infty w\Big(\frac{u}{x}\Big)w\Big(\frac{v}{x}\Big)e\Big(-\frac{\tilde nu+\tilde mv+2\sqrt{uv}\sin\theta}{c}\Big) dudvd\theta.
\end{multline*}
Consider the innermost integral over \(u\). Writing \(F(u)=F(u,v,\theta)=\tilde nu+2\sqrt{uv}\sin\theta\), this is
\[\int_0^\infty w\Big(\frac{u}{x}\Big) e\Big(-\frac{F(u)}{c}\Big)du.\]
Note that for any \(\theta\) and \(x\leq u,v\leq 2x\) in the range of integration, one has 
\[\frac{\partial F}{\partial u}=\tilde n+\sqrt{\frac{v}{u}}\sin\theta \implies \Big|\frac{\partial F}{\partial u}\Big|\geq |\tilde n|-\sqrt{2}.\]
Moreover, one observes that for \(u,v\) in this range
\[\frac{\partial^jF}{\partial u^j}\ll_j (|\tilde n|+1) x^{1-j}, \text{ and } \frac{\partial^j}{\partial u^j}w\Big(\frac{u}{x}\Big)\ll_j \Big(\frac{k^{1-\epsilon}}{x}\Big)^{j} \text{ for } j=0,1,2,\ldots\]
(Recall \(w^{(j)}\ll_j \Delta^{j}=k^{(1-\epsilon)j}\) by our construction.) Consequently if \(|\tilde n|\geq2\), we apply Lemma \ref{bkyibp} (repeated integration by parts) with \(X=1\), \(R=|\tilde n|/c\), \(U=xk^{-1+\epsilon}\), \(Y=|\tilde n|x/c\) and \(Q=x\) to obtain the bound 
\begin{equation}\label{removey} 
\int_x^{2x} w\Big(\frac{u}{x}\Big) e\Big(-\frac{F(u)}{c}\Big)du\ll_B x\Big\{\Big(\frac{|\tilde n|x}{c}\Big)^{-B/2}+\Big(\frac{|\tilde n|x}{ck^{1-\epsilon}}\Big)^{-B}\Big\},
\end{equation}
valid for any integer \(B\geq 0\). 
Note that for \(c\leq 32\pi x/k\) (in the range of summation), \(x/c\geq x/(ck^{1-\epsilon})\geq k^\epsilon/(32\pi)\). 
So taking \(B\) large enough in (\ref{removey}), provided \(|\tilde n|\geq 2\) we have
\[\int_0^\infty w\Big(\frac{u}{x}\Big) e\Big(-\frac{F(u)}{c}\Big)du\ll \tilde n^{-2}k^{-1100}, \text{ say}.\]
As a result, the contribution to (\ref{offdiag101v3}) of all terms with \(|\tilde n|\geq2\) or (by symmetry) \(|\tilde m|\geq2\) is seen to be \(\ll k^{-1000}\). So we may rewrite (\ref{offdiag101v3}) as 
\begin{multline}\label{offdiag101v4}
(\mathrm{OD})=2\pi (-1)^{k/2}\sum_{c\leq 32\pi x/k} c^{-1}\sum_{\substack{1\leq a\leq c\\(a,c)=1}}
\sum_{\substack{\tilde n,\tilde m\in\{-1,0,1\}\\\tilde n\equiv -a\Mod c\\\tilde m\equiv -a^*\Mod c}}\\
\int_0^\infty \int_0^\infty w\Big(\frac{u}{x}\Big)w\Big(\frac{v}{x}\Big)J_{k-1}\Big(\frac{4\pi\sqrt{uv}}{c}\Big)e\Big(-\frac{\tilde nu+\tilde mv}{c}\Big)dudv+\mathcal O(k^{-1000}).
\end{multline}

Finally, we check the ways in which the conditions \(\tilde n,\tilde m\in\{-1,0,1\}\), \(\tilde n\equiv -a\Mod c\) and \(\tilde m \equiv -a^*\Mod c\) can be met, given any \(c\geq1\) and \(1\leq a\leq c\) with \((a,c)=1\). If \(c=1\) then necessarily \(a=a^*=1\). So we must consider the contribution from all possibilities
\[(\tilde n,\tilde m)\in \Gamma_1=\{(-1,-1),(-1,0),(-1,1),(0,-1),(0,0),(0,1),(1,-1),(1,0),(1,1)\}.\]
If \(c=2\), then once again \(a=a^*=1\) is forced, so we require \(\tilde n,\tilde m\equiv 1\Mod 2\). Hence there is a possible contribution from all
\[(\tilde n,\tilde m)\in\Gamma_2=\{(-1,-1),(-1,1),(1,-1),(1,1)\}.\]
If \(c\geq 3\), then \(\tilde n\equiv -a\Mod c\) and \(\tilde n\in\{-1,0,1\}\) implies \(a\equiv a^*\equiv \pm1\Mod c\). Therefore \(\tilde m\equiv -a^*\equiv \tilde n\Mod c\), and hence \(\tilde n=\tilde m\) (as \(|\tilde n|,|\tilde m|\leq 1\) and \(c\geq 3\)). So when \(c\geq 3\), the only possibilities to consider are 
\[(\tilde n,\tilde m)\in\Gamma_c=\{(-1,-1),(1,1)\}.\]
Restricting the summation in (\ref{offdiag101v4}) to only these values of \((\tilde n,\tilde m)\) gives the lemma.
\end{proof}

In fact, with a little more work it is possible to further truncate the sum in Lemma \ref{initialpoisson}. 

\begin{lemma}\label{initalpoissonsimplified}
We have
\begin{multline*}
(\mathrm{OD})=4\pi (-1)^{k/2}x^2\int_0^\infty yJ_{k-1}(4\pi xy)\int_0^\infty w\Big(\frac{y^2}{z}\Big)\frac{w(z)}{z}dzdy\\
+4\pi (-1)^{k/2} x^2\sum_{c\leq 32\pi x/k}c^{-1}\sum_{\varepsilon=\pm1}\int_0^\infty yJ_{k-1}\Big(\frac{4\pi xy}{c}\Big) \int_0^\infty w\Big(\frac{y^2}{z}\Big)\frac{w(z)}{z} e\Big(\frac{\varepsilon x}{c}\Big(z+\frac{y^2}{z}\Big)\Big)dzdy\\
+\mathcal O(k^{-1000}).
\end{multline*}
\end{lemma}

\begin{proof}
We consider the integrals
\[\int_0^\infty\int_0^\infty w\Big(\frac{u}{x}\Big)w\Big(\frac{v}{x}\Big)J_{k-1}\Big(\frac{4\pi\sqrt{uv}}{c}\Big)e\Big(\frac{\tilde nu+\tilde mv}{c}\Big)dudv\]
appearing in Lemma \ref{initialpoisson}. It is convenient to apply the change of variables \(y=\sqrt{uv}/x\) and \(z=v/x\), which shows that the above is equal to 
\[2x^2\int_0^\infty yJ_{k-1}\Big(\frac{4\pi xy}{c}\Big) \int_0^\infty w\Big(\frac{y^2}{z}\Big)\frac{w(z)}{z} e\Big(\frac{x}{c}\Big(\tilde mz+\frac{\tilde ny^2}{z}\Big)\Big)dzdy.\]
Write \(G(z)=G_y(z)=w(y^2/z)w(z)/z\) and \(F(z)=F_y(z;\tilde m, \tilde n)=\tilde mz+\tilde ny^2/z\), and consider the innermost oscillatory integral
\begin{equation*}
\int_0^\infty w\Big(\frac{y^2}{z}\Big)\frac{w(z)}{z} e\Big(\frac{x}{c}\Big(\tilde mz+\frac{\tilde ny^2}{z}\Big)\Big)dz=\int_{\max(1,y^2/2)}^{\min(2,y^2)} G(z)e\Big(\frac{xF(z)}{c}\Big)dz.
\end{equation*}
For \(z\) in the region of integration (i.e. \(\max(1,y^2/2)\leq z\leq \min(2, y^2)\)), we have
\[F'(z)=\tilde m-\tilde ny^2/z^2 \text{ and } G^{(j)}(z)\ll_j k^{(1-\epsilon)j} \text{ for any integer } j\geq0.\]
(Recall \(w^{(j)}\ll_j \Delta^{j}=k^{(1-\epsilon)j}\) by our construction.) Consequently, if 
\begin{equation}\label{bconds}
(\tilde n,\tilde m)\in\{(-1,0),(1,0),(0,-1),(0,1),(-1,1),(1,-1)\},
\end{equation}
and \(z, y^2/z \in\supp w=[1,2]\) (as in the region of integration), then \(|F'(z)|=|\tilde m -\tilde ny^2/z^2|\geq 1/2\). It is also clear that \(F^{(j)}(z)\ll_j 1\) for all integers \(j\geq 0\) (and \(z, y^2/z\in[1,2]\) and \(|\tilde n|, |\tilde m|\leq 1\) as above). We may thus apply Lemma \ref{bkyibp} (with \(X=1\), \(U=k^{-1+\epsilon}\), \(R=x/c\), \(Y=x/c\) and \(Q=1\)) to obtain
\begin{equation}\label{oscillint2}
\int_{\max(1,y^2/2)}^{\min(2,y^2)} G(z)e\Big(\frac{xF(z)}{c}\Big)dz\ll_B \Big(\frac{x}{c}\Big)^{-B/2}+\Big(\frac{x}{ck^{1-\epsilon}}\Big)^{-B} \text{ valid for any integer } B\geq0.
\end{equation}
Since \(c\leq 32\pi x/k\) (in the range of summation in Lemma \ref{initialpoisson}), \(x/c\geq x/(ck^{1-\epsilon})\geq k^\epsilon/(32\pi)\). So upon taking \(B\) large enough in (\ref{oscillint2}), provided \(\tilde n\) and \(\tilde m\) satisfy (\ref{bconds}) we have the bound
\begin{multline*} 
\int_0^\infty\int_0^\infty w\Big(\frac{u}{x}\Big)w\Big(\frac{v}{x}\Big)J_{k-1}\Big(\frac{4\pi\sqrt{uv}}{c}\Big)e\Big(\frac{\tilde nu+\tilde mv}{c}\Big)dudv\\
=2x^2\int_0^\infty yJ_{k-1}\Big(\frac{4\pi xy}{c}\Big) \int_{\max(1,y^2/2)}^{\min(2,y^2)} G(z)e\Big(\frac{xF(z)}{c}\Big)dzdy\ll k^{-1000}.
\end{multline*}
It follows (see Lemma \ref{initialpoisson}) that the contribution to \((\mathrm{OD})\) of those \(\tilde n, \tilde m\) satisfying (\ref{bconds}) is \(\ll k^{-1000}\). So the only terms in Lemma \ref{initialpoisson} that contribute a (possibly) non-negligible amount to \((\mathrm{OD})\) are the \(c=1\) term with \((\tilde n,\tilde m)=(0,0)\) and all the \(c\geq1\) terms with \((\tilde n,\tilde m)=\pm(1,1)\). Restricting the summation in Lemma \ref{initialpoisson} to only these values of \(c, \tilde n\) and \(\tilde m \), and applying the change of variables \(y=\sqrt{uv}/x\), \(z=v/x\) as above completes the proof.
\end{proof}

\subsubsection{Evaluating the Oscillatory Integrals}

Equipped with Lemma \ref{initalpoissonsimplified}, we are now in a position to explicitly evaluate the off-diagonal contribution. The first integral in Lemma \ref{initalpoissonsimplified} (corresponding to \(c=1\) and \((\tilde n, \tilde m)=(0,0)\)) will contribute the secondary main term arising in part \ref{varii} of Theorem \ref{thmvar}. The other terms (corresponding to \((\tilde n,\tilde m)=\pm(1,1)\)) will exhibit some cancellation (owing to the extra exponential terms), and contribute only an error term.

In order to evaluate and bound the integrals appearing in Lemma \ref{initalpoissonsimplified}, it is convenient to unsmooth the integrals -- recall \(w=w_\Delta\) is as given in (\ref{wdef}) with \(\Delta=k^{1-\epsilon}\). The unsmoothing is done in the following lemma.

\begin{lemma}\label{lemunsmooth}
We have
\begin{multline*}
(\mathrm{OD})=\frac{(-1)^{k/2}}{2\pi}\int_0^\infty yL\Big(\frac{y}{4\pi x}\Big)J_{k-1}(y)dy\\
+\frac{(-1)^{k/2}}{2\sqrt{2\pi}}\sum_{c\leq 32\pi x/k}c\sum_{\varepsilon=\pm1} e(\varepsilon/8)\int_{4\pi x/c}^{8\pi x/c} y^{1/2}J_{k-1}(y)e^{i\varepsilon y}dy+\mathcal O(k^{2/3+4\epsilon}).
\end{multline*}
\end{lemma}

\begin{proof}
Write \(\mathcal A=\{(y,z)\in\mathbb{R}^2: z, y^2/z\in[1,2]\}=\{(z,y)\in[1,2]^2: \max(1,y^2/2)\leq z\leq \min(2, y^2)\}\) and \(\mathcal A_\Delta=\{(y,z)\in\mathcal A:z \text{ or } y^2/z\in[1, 1+\Delta^{-1}]\cup[2-\Delta^{-1}, 2]\}\) (with \(\Delta=k^{1-\epsilon}\)). Considering the first integral in Lemma \ref{initalpoissonsimplified}, one has (using the properties of \(w\) given in (\ref{wdef}))
\begin{multline*}
\int_0^\infty yJ_{k-1}(4\pi xy)\int_0^\infty w\Big(\frac{y^2}{z}\Big)\frac{w(z)}{z}dzdy\\
=\iint_{\mathcal A} \frac yz J_{k-1}(4\pi xy) dzdy+\mathcal O\Big(\iint_{\mathcal A_\Delta}\Big|\frac yz J_{k-1}(4\pi xy)\Big|dzdy\Big).
\end{multline*}
Using the bound (\ref{bessel2}) of Lemma \ref{besselbounds} and the fact \(\mathrm{vol}(\mathcal A_\Delta)\ll \Delta^{-1}\), we bound the above error term by
\[\iint_{\mathcal A_\Delta}\Big|\frac yz J_{k-1}(4\pi xy)\Big|dzdy\ll k^{-1/3}\mathrm{vol}(\mathcal A_\Delta)\ll k^{-1/3}\Delta^{-1}=k^{-4/3+\epsilon}.\]
Therefore 
\begin{multline}\label{mainterm}
\int_0^\infty yJ_{k-1}(4\pi xy)\int_0^\infty w\Big(\frac{y^2}{z}\Big)\frac{w(z)}{z}dzdy=\int_1^2y J_{k-1}(4\pi xy) \int_{\max(1,y^2/2)}^{\min(2,y^2)}\frac{1}{z}dzdy +\mathcal O (k^{-4/3+\epsilon})\\
=2\int_1^{\sqrt{2}} y\log y J_{k-1}(4\pi xy)dy +2\int_{\sqrt2}^2 y\log(2/y) J_{k-1}(4\pi xy)dy+\mathcal O (k^{-4/3+\epsilon})\\
=2\int_0^\infty yL(y)J_{k-1}(4\pi xy) dy +\mathcal O (k^{-4/3+\epsilon}).
\end{multline}

Similarly, the second integral in Lemma \ref{initalpoissonsimplified} is
\begin{multline}\label{errortermscgeq1}
\int_0^\infty yJ_{k-1}\Big(\frac{4\pi xy}{c}\Big) \int_0^\infty w\Big(\frac{y^2}{z}\Big)\frac{w(z)}{z} e\Big(\frac{\varepsilon x}{c}\Big(z+\frac{y^2}{z}\Big)\Big)dzdy\\
=\int_1^2y J_{k-1}\Big(\frac{4\pi xy}{c}\Big)\int_{\max(1,y^2/2)}^{\min(2,y^2)} \frac{1}{z}e\Big(\frac{\varepsilon x}{c}\Big(z+\frac{y^2}{z}\Big)\Big)dzdy +\mathcal O (k^{-4/3+\epsilon}).
\end{multline}
Let \(F(z)=z+y^2/z\). Observe \(z=y\) is the only zero of \(F'(z)=1-y^2/z^2\). Thus the standard stationary phase expansion given in e.g. \cite[Theorem 2.2]{ivic} shows
\begin{multline*}
\int_{\max(1,y^2/2)}^{\min(2,y^2)} \frac{1}{z}e\Big(\frac{\varepsilon x}{c}F(z)\Big)dz
=\frac{e(\varepsilon/8)}{\sqrt 2} \Big(\frac{x}{c}\Big)^{-1/2} y^{-1/2}e\Big(\frac{2\varepsilon xy}{c}\Big)+\mathcal O\Big(\Big(\frac{x}{c}\Big)^{-3/2}\Big)\\
+\mathcal O\Big(\Big(\frac{x}{c}\Big)^{-1/2}\min\Big\{1, \Big(\frac{x}{c}\Big)^{-1/2}\frac{1}{|\max(1,y^2/2)-y|}\Big\}\Big)\\
+\mathcal O\Big(\Big(\frac{x}{c}\Big)^{-1/2} \min\Big\{1, \Big(\frac{x}{c}\Big)^{-1/2}\frac{1}{|\min(2,y^2)-y|}\Big\}\Big).
\end{multline*}
Applying this to the expression in (\ref{errortermscgeq1}) and using the bound \(J_{k-1}(4\pi xy/c)\ll k^{-1/3}\) (which is (\ref{bessel2}) of Lemma \ref{besselbounds}), we have
\begin{multline}\label{errortermscgeq1v2}
\int_1^2y J_{k-1}\Big(\frac{4\pi xy}{c}\Big)\int_{\max(1,y^2/2)}^{\min(2,y^2)} \frac{1}{z}e\Big(\frac{\varepsilon x}{c}\Big(z+\frac{y^2}{z}\Big)\Big)dzdy\\
=\frac{e(\varepsilon/8)}{\sqrt2} \Big(\frac{x}{c}\Big)^{-1/2}\int_1^{2} y^{1/2}J_{k-1}\Big(\frac{4\pi xy}{c}\Big)e\Big(\frac{2\varepsilon xy}{c}\Big)dy+\mathcal O\Big(k^{-1/3}\Big(\frac xc\Big)^{-3/2}\Big)\\
+\mathcal O\Big(k^{-1/3}\Big(\frac{x}{c}\Big)^{-1/2}\int_1^{2}\Big(\min\Big\{1, \Big(\frac{x}{c}\Big)^{-1/2}\frac{1}{|\max(1,y^2/2)-y|}\Big\}\\
+\min\Big\{1, \Big(\frac{x}{c}\Big)^{-1/2}\frac{1}{|\min(2,y^2)-y|}\Big\}\Big)dy\Big).
\end{multline}
The final error term above is
\begin{align*}
&\ll k^{-1/3}\Big(\frac xc\Big)^{-1/2}\Big(\int_1^{\sqrt2} \min\Big\{1, \Big(\frac xc\Big)^{-1/2} \frac{1}{ y-1}\Big\}dy+\int_{\sqrt2}^2 \min\Big\{1, \Big(\frac xc\Big)^{-1/2} \frac{1}{2-y}\Big\}dy\Big)\\
&\ll  k^{-1/3}\Big(\frac xc\Big)^{-1/2}\Big(\int_1^{1+(x/c)^{-1/2}}dy+\Big(\frac xc\Big)^{-1/2}\int_{1+(x/c)^{-1/2}}^{\sqrt2} \frac {dy}{y-1}\\
&\hspace{4em}+\Big(\frac xc\Big)^{-1/2}\int_{\sqrt2}^{2-(x/c)^{-1/2}}\frac{dy}{2-y}+\int_{2-(x/c)^{-1/2}}^2dy\Big) \ll k^{-1/3}\Big(\frac xc\Big)^{-1}\log \Big(\frac xc\Big)\ll k^{-4/3+\epsilon}.
\end{align*}
(In the last step, we used that \(c\leq 32\pi x/k\).) We conclude from (\ref{errortermscgeq1}) and (\ref{errortermscgeq1v2}) that 
\begin{multline}\label{secondintexpression}
\int_0^\infty yJ_{k-1}\Big(\frac{4\pi xy}{c}\Big) \int_0^\infty w\Big(\frac{y^2}{z}\Big)\frac{w(z)}{z} e\Big(\frac{\varepsilon x}{c}\Big(z+\frac{y^2}{z}\Big)\Big)dzdy\\
=\frac{e(\varepsilon/8)}{\sqrt 2} \Big(\frac{x}{c}\Big)^{-1/2}\int_1^2 y^{1/2}J_{k-1}\Big(\frac{4\pi xy}{c}\Big)e\Big(\frac{2\varepsilon xy}{c}\Big)dy+\mathcal O(k^{-4/3+\epsilon}).
\end{multline}

Replacing the estimates (\ref{mainterm}) and (\ref{secondintexpression}) in Lemma \ref{initalpoissonsimplified}, we obtain
\begin{multline*}\label{offdiag106}
(\mathrm{OD})=8\pi (-1)^{k/2} x^2\int_0^\infty yL(y)J_{k-1}(4\pi xy)dy+\mathcal O(x^2k^{-4/3+\epsilon})\\
+2\sqrt{2}\pi (-1)^{k/2} x^{3/2}\sum_{c\leq 32\pi x/k}c^{-1/2}\sum_{\varepsilon=\pm1}e(\varepsilon/8)\int_1^2 y^{1/2}J_{k-1}\Big(\frac{4\pi xy}{c}\Big)e\Big(\frac{2\varepsilon xy}{c}\Big)dy\\
+\mathcal O\Big(x^2k^{-4/3+\epsilon}\log\Big(\frac {32\pi x}{k}+1\Big)\Big).
\end{multline*}
Using the standing assumption \(x\leq k^{1+\epsilon}\), it follows that both error terms above are \(\mathcal O(k^{2/3+4\epsilon})\). The lemma now follows upon making the change of variables \(y\mapsto y/(4\pi x)\) in the first integral above, and \(y\mapsto cy/(4\pi x)\) in the remaining integrals.
\end{proof}

The remainder of the proof of part \ref{varii} of Theorem \ref{thmvar} is split into two lemmas. The first of these gives an asymptotic for the main term (the first integral in Lemma \ref{lemunsmooth}), and the second gives a bound for the error terms. For the rest of this section, we will denote \(\kappa=k-1\) where it is convenient. 

\begin{lemma}[Main Term]\label{maintermlem}
We have
\begin{equation}\label{mt1}
\int_0^\infty yL\Big(\frac{y}{4\pi x}\Big)J_{k-1}(y)dy=L\Big(\frac{k}{4\pi x}\Big)k(1+\mathcal O(k^{-1/8}))+\mathcal O(k^{2/3}).
\end{equation}
\end{lemma}

\begin{proof}
Using (\ref{bessel1}) to bound the contribution of \(y<k-k^{1/2}\) to the integral (\ref{mt1}), we write
\begin{equation}\label{maintermsplit}
\int_0^\infty yL\Big(\frac{y}{4\pi x}\Big)J_{k-1}(y)dy=\Big\{\int_{k-k^{1/2}}^{k+k^{1/2}}+\int_{k+k^{1/2}}^\infty \Big\}yL\Big(\frac{y}{4\pi x}\Big)J_{k-1}(y)dy+\mathcal O(k^2e^{-k^{1/6}}).
\end{equation}
Since \(|L'(\xi)|\leq 1/\xi\) for all \(\xi\) and \(\supp L=[1,2]\), the mean value theorem shows
\begin{equation}\label{Lapprox} 
L\Big(\frac{y}{4\pi x}\Big)=L\Big(\frac{k}{4\pi x}\Big)+\mathcal O\Big(\frac{|y-k|}{k}\Big).
\end{equation}
Consequently,
\begin{equation*} 
\int_{k-k^{1/2}}^{k+k^{1/2}}yL\Big(\frac{y}{4\pi x}\Big)J_{k-1}(y)dy
=L\Big(\frac{k}{4\pi x}\Big)\int_{k-k^{1/2}}^{k+k^{1/2}}yJ_{k-1}(y)dy+\mathcal O\Big(k^{-1/2}\int_{k-k^{1/2}}^{k+k^{1/2}}|yJ_{k-1}(y)|dy\Big).
\end{equation*}
Using the bound (\ref{bessel2}) of Lemma \ref{besselbounds}, the error term above is seen to be \(\mathcal O(k^{2/3})\). We now apply Lemma \ref{transrangeint} to the main term above, which yields
\begin{equation}\label{transitionrangey}
\int_{k-k^{1/2}}^{k+k^{1/2}}yL\Big(\frac{y}{4\pi x}\Big)J_{k-1}(y)dy=L\Big(\frac{k}{4\pi x}\Big)k(1+\mathcal O(k^{-1/8}))+\mathcal O(k^{2/3}).
\end{equation}
Replacing (\ref{transitionrangey}) in (\ref{maintermsplit}), we have shown
\begin{multline}\label{maintermsplitv2}
\int_0^\infty yL\Big(\frac{y}{4\pi x}\Big)J_{k-1}(y)dy\\
=L\Big(\frac{k}{4\pi x}\Big)k(1+\mathcal O(k^{-1/8}))+\int_{k+k^{1/2}}^\infty yL\Big(\frac{y}{4\pi x}\Big) J_{k-1}(y)dy+\mathcal O(k^{2/3}).
\end{multline}

We now bound the remaining integral in (\ref{maintermsplitv2}), which is taken over the oscillatory regime of the Bessel function. Note this integral appears only if \(8\pi x\geq k+k^{1/2}\). We first replace the Bessel function by the asymptotic (\ref{bessel3trunc}) of Lemma \ref{besselasymptotics}, which shows (recall \(\kappa=k-1\))
\begin{multline}\label{asympreplacedagain}
\int_{k+k^{1/2}}^\infty yL\Big(\frac{y}{4\pi x}\Big)J_{k-1}(y)dy=\sqrt{\frac 2\pi} \int_{k+k^{1/2}}^\infty y(y^2-\kappa^2)^{-1/4}L\Big(\frac{y}{4\pi x}\Big)\cos\omega(y)dy\\
+\mathcal O\Big(\int_{k+k^{1/2}}^\infty y^3 (y^2-\kappa^2)^{-7/4}L\Big(\frac{y}{4\pi x}\Big)dy\Big).
\end{multline}
Here \(\omega(y)=\omega_\kappa(y)\) is as given in (\ref{omega}). Using (\ref{omega'}), which states \(\omega'(y)=(y^2-\kappa^2)^{1/2}/y\), we integrate by parts to see that the first integral above is
\begin{multline*}
\int_{k+k^{1/2}}^\infty y(y^2-\kappa^2)^{-1/4}L\Big(\frac{y}{4\pi x}\Big)\cos\omega(y)dy=\int_{k+k^{1/2}}^\infty y^2 (y^2-\kappa^2)^{-3/4} L\Big(\frac{y}{4\pi x}\Big)d(\sin\omega(y))\\
=-(k+k^{1/2})^2((k+k^{1/2})^2-\kappa^2)^{-3/4}L\Big(\frac{k+k^{1/2}}{4\pi x}\Big)\sin\omega(k+k^{1/2})\\
-\int_{k+k^{1/2}}^\infty \frac{d}{dy}\Big\{y^2(y^2-\kappa^2)^{-3/4}L\Big(\frac{y}{4\pi x}\Big)\Big\}\sin\omega(y)dy.
\end{multline*}
Using \(|L'(\xi)|\leq 1/\xi\), \(|L(\xi)|\leq \log 2/2\), (\ref{Lapprox}), and the fact that \(\supp L=[1,2]\), the above is seen to be
\begin{multline*}
\ll k^{7/8}\Big(L\Big(\frac{k}{4\pi x}\Big)+\mathcal O(k^{-1/2})\Big)+\int_{\max(k+k^{1/2}, 4\pi x)}^{8\pi x} y (y^2-k^2)^{-3/4}dy\\
+\int_{k+k^{1/2}}^\infty y^3(y^2-k^2)^{-7/4}L\Big(\frac{y}{4\pi x}\Big)dy.
\end{multline*}
Combining this with (\ref{asympreplacedagain}), we have shown
\begin{multline}\label{asympreplacedbutnowibp}
\int_{k+k^{1/2}}^\infty yL\Big(\frac{y}{4\pi x}\Big)J_{k-1}(y)dy\ll L\Big(\frac{k}{4\pi x}\Big) k^{7/8}+k^{3/8}\\
+\int_{\max(k+k^{1/2}, 4\pi x)}^{8\pi x} y (y^2-k^2)^{-3/4}dy
+\int_{k+k^{1/2}}^\infty y^3(y^2-k^2)^{-7/4}L\Big(\frac{y}{4\pi x}\Big)dy.
\end{multline}
The former integral in (\ref{asympreplacedbutnowibp}) is
\begin{equation}\label{eqint1} 
\int_{\max(k+k^{1/2}, 4\pi x)}^{8\pi x} y (y^2-k^2)^{-3/4}dy=2(y^2-k^2)^{1/4}\Big|_{\max(k+k^{1/2}, 4\pi x)}^{8\pi x}\ll x^{1/2}.
\end{equation}
Using (\ref{Lapprox}) we write the latter integral in (\ref{asympreplacedbutnowibp}) as
\begin{equation*}
\int_{k+k^{1/2}}^\infty y^3(y^2-k^2)^{-7/4}L\Big(\frac{y}{4\pi x}\Big)dy=\int_{\max(k+k^{1/2},4\pi x)}^{8\pi x} y^3(y^2-k^2)^{-7/4}\Big(L\Big(\frac{k}{4\pi x}\Big)+\mathcal O\Big(\frac{y-k}{k}\Big)\Big)dy.
\end{equation*}
Now
\begin{multline*}
L\Big(\frac{k}{4\pi x}\Big)\int_{\max(k+k^{1/2},4\pi x)}^{8\pi x} y^3(y^2-k^2)^{-7/4}dy\\
=L\Big(\frac{k}{4\pi x}\Big)\Big[2(y^2-k^2)^{1/4}-\frac23 k^2(y^2-k^2)^{-3/4}\Big]_{\max(k+k^{1/2}, 4\pi x)}^{8\pi x}\ll L\Big(\frac{k}{4\pi x}\Big) k^{7/8}+x^{1/2},
\end{multline*}
and 
\begin{equation*}
k^{-1}\int_{\max(k+k^{1/2},4\pi x)}^{8\pi x} y^3(y-k)(y^2-k^2)^{-7/4}dy\ll k^{-1}\int_{\max(k+k^{1/2},4\pi x)}^{8\pi x}y^{5/4}(y-k)^{-3/4}dy\ll k^{-1}x^{3/2}.
\end{equation*}
So overall, the latter integral in (\ref{asympreplacedbutnowibp}) is 
\begin{equation}\label{eqint2}
\int_{k+k^{1/2}}^\infty y^3(y^2-k^2)^{-7/4}L\Big(\frac{y}{4\pi x}\Big)dy\ll L\Big(\frac{k}{4\pi x}\Big)k^{7/8}+x^{1/2}+k^{-1}x^{3/2}.
\end{equation}
Replacing the two estimates (\ref{eqint1}) and (\ref{eqint2}) in (\ref{asympreplacedbutnowibp}), we obtain (using the standing assumption \(x\leq k^{1+\epsilon}\))
\[\int_{k+k^{1/2}}^\infty yL\Big(\frac{y}{4\pi x}\Big)J_{k-1}(y)dy\ll L\Big(\frac{k}{4\pi x}\Big)k^{7/8}+x^{1/2}+k^{-1}x^{3/2}\ll L\Big(\frac{k}{4\pi x}\Big)k^{7/8}+k^{1/2+3\epsilon/2}. \]
The lemma now follows from (\ref{maintermsplitv2}).
\end{proof}

\begin{lemma}[Error Terms]\label{errortermlem}
For \(c\leq 32\pi x/k
\) and \(\varepsilon=\pm1\), we have
\begin{equation*}
\int_{4\pi x/c}^{8\pi x/c}y^{1/2}J_{k-1}(y)e^{i\varepsilon y}dy\ll k^{2/3}.
\end{equation*}
\end{lemma}

\begin{proof}
Using (\ref{bessel1}) of Lemma \ref{besselbounds}, we see that the contribution of \(y\in[4\pi x/c,8\pi x/c]\cap[0,k-k^{1/2}]\) is negligible, since \(J_{k-1}(y)\ll e^{-k^{1/6}}\) in this range. Moreover, using (\ref{bessel2}) we can bound the contribution of \(y\in[4\pi x/c,8\pi x/c]\cap [k-k^{1/2},k+k^{1/2}]\) by
\[\int_{k-k^{1/2}}^{k+k^{1/2}}|y^{1/2}J_{k-1}(y)|dy\ll \int_{k-k^{1/2}}^{k+k^{1/2}}k^{1/2}\cdot k^{-1/3}dy\ll k^{2/3},\]
which is acceptable.

To prove the lemma, it therefore suffices to show that the contribution from  \(y\in[k+k^{1/2},\infty)\cap [4\pi x/c,8\pi x/c]\) (which appears only if \(8\pi x/c\geq k+k^{1/2}\)) is \(\mathcal O(k^{2/3})\). Applying the asymptotic (\ref{bessel3trunc}) for the Bessel functions (valid in this range of of \(y\)), we find that this contribution is 
\begin{align}\label{oscillcontrib}
&\int_{\max(k+k^{1/2}, 4\pi x/c)}^{8\pi x/c} y^{1/2}J_{k-1}(y)e^{i\varepsilon y}dy\\
=&\sqrt{\frac2\pi}\int_{\max(k+k^{1/2}, 4\pi x/c)}^{8\pi x/c} y^{1/2}(y^2-\kappa^2)^{-1/4}\cos\omega(y)e^{i\varepsilon y}dy\nonumber\\
&\hspace{10em}+\mathcal O\Big(\int_{\max(k+k^{1/2}, 4\pi x/c)}^{8\pi x/c}y^{5/2}(y^2-k^2)^{-7/4}dy\Big)\nonumber\\
\ll& \Big|\int_{\max(k+k^{1/2}, 4\pi x/c)}^{8\pi x/c} y^{1/2}(y^2-\kappa^2)^{-1/4}(e^{i(\varepsilon y-\omega(y))}+e^{i(\varepsilon y+\omega(y))})dy\Big|+\Big(\frac xc\Big)^{3/2}k^{-9/8}.\nonumber
\end{align}
It is a standard matter to bound the oscillatory integrals above. First (recalling (\ref{omega'})) note
\begin{equation*}
\frac{d}{dy}\{\varepsilon y\pm \omega(y)\}=\varepsilon \pm\Big(1-\frac{\kappa^2}{y^2}\Big)^{1/2}.
\end{equation*}
So we obtain 
\begin{equation*}
\Big|\frac{d}{dy}\{\varepsilon y\pm\omega(y)\}\Big|\geq 1-\Big(1-\frac{\kappa^2}{y^2}\Big)^{1/2}\geq \frac{\kappa^2}{2y^2}\gg\frac{k^2}{y^2}.
\end{equation*}
We now apply \cite[Lemma 4.4]{titchmarsh} (possibly after partitioning the interval of integration into two subintervals on which \((\varepsilon \pm\omega'(y))(y^2-\kappa^2)^{1/4}y^{-1/2}\) is monotonic), which gives the bound
\begin{multline*}
\int_{\max(k+k^{1/2}, 4\pi x/c)}^{8\pi x/c} y^{1/2}(y^2-\kappa^2)^{-1/4}e^{i(\varepsilon y\pm \omega(y))}dy\\
\ll y^{5/2}k^{-2}(y^2-\kappa^2)^{-1/4}\Big|_{y=\max(k+k^{1/2}, 4\pi x/c)}\ll k^{1/8}.
\end{multline*}
We now conclude from (\ref{oscillcontrib}) that the contribution of \(y\in[k+k^{1/2},\infty)\cap [4\pi x/c,8\pi x/c]\) is \(\ll k^{1/8}+x^{3/2}k^{-9/8}\ll k^{3/8+3\epsilon/2}\) (by the standing assumption \(x\leq k^{1+\epsilon}\)), which is acceptable. This proves the lemma.
\end{proof}

\begin{proof}[Proof of part \ref{varii} of Theorem \ref{thmvar}]
Applying Lemmas \ref{maintermlem} \& \ref{errortermlem} to Lemma \ref{lemunsmooth}, we obtain
\begin{multline*}
(\mathrm{OD})=\frac{(-1)^{k/2}}{2\pi} L\Big(\frac{k}{4\pi x}\Big)k(1+\mathcal O(k^{-1/8})) +\mathcal O(k^{2/3}) + \mathcal O\Big(k^{2/3}\sum_{c\leq 32\pi x/k}c\Big) +\mathcal O(k^{2/3+4\epsilon})\\
=\frac{(-1)^{k/2}}{2\pi} L\Big(\frac{k}{4\pi x}\Big)k(1+\mathcal O(k^{-1/8})) +\mathcal O(k^{2/3+4\epsilon}).
\end{multline*}
In the last line, we used the standing assumption \(x\leq k^{1+\epsilon}\) (to bound \(\sum_c c\ll k^{2\epsilon}\)). Equipped with this estimate for the off-diagonal terms (which also implies \((\mathrm{OD})\ll x\)), Lemma \ref{insplit4lem} shows
\begin{equation*} 
\langle \mathcal S(x,f)^2\rangle =x+\frac{(-1)^{k/2}}{2\pi} L\Big(\frac{k}{4\pi x}\Big)k(1+\mathcal O(k^{-1/8})) +\mathcal O(k^{2/3+4\epsilon})+\mathcal O(x^{1/2}k^{3\epsilon}),
\end{equation*}
completing the proof.
\end{proof}

\subsection{Proof of Theorem \ref{thmvar}, part \ref{variii}} \label{subsecpart3}

Throughout this section (§\ref{subsecpart3}), we fix \(\epsilon\) and assume \( k^{1+\epsilon}\leq x\leq k^2\). The key idea to handle \(\langle \mathcal S(x,f)^2\rangle\) in this range is to transform \(\mathcal S(x,f)\) using Lemma \ref{vorprop}. This shortens the effective length of the sums, after which the off-diagonal terms (arising from the trace formula) will be seen to contribute only a negligible quantity (as in the proof of part \ref{vari} of Theorem \ref{thmvar}). 

Lemma \ref{vorprop} introduces a smoothing function \(w=w_\Delta\) (given in (\ref{wdef})), and the transformed sums contain a Hankel-type transform of \(w\) given by 
\begin{equation}\label{tildewdefrepeated}
\tilde w(\xi)=\tilde w_\Delta(\xi)=\int_0^\infty w_\Delta(t)J_{k-1}(4\pi \sqrt{(k^2+\Delta^2)\xi t})dt.
\end{equation}
Evaluating the diagonal terms arising from the Petersson formula now requires some analysis of these transforms -- the key input will be the following lemma. 

\begin{lemma}\label{diagterms}
Let \(\Delta\) be a sufficiently large parameter satisfying \(\Delta\leq k\), and let \(\tilde w=\tilde w_\Delta\) as given in (\ref{tildewdefrepeated}). Then
\begin{equation*}
4\pi^2 x^2 \sum_{n\geq 1} \tilde w\Big(\frac{nx}{k^2+\Delta^2}\Big)^2=x+\mathcal O\Big(\frac x\Delta\Big)+\mathcal O\Big(\frac{x^{3/2}\log^3 k}{k}\Big).
\end{equation*}
\end{lemma}

\begin{proof}[Proof of part \ref{variii} of Theorem \ref{thmvar}, assuming Lemma \ref{diagterms}]
For any sufficiently large parameter \(\Delta\leq k\leq x^{1-\epsilon/2}\), Lemma \ref{vorprop} gives that for \(f\in\mathcal B_k\), 
\begin{equation}\label{vorpropexpl}
\mathcal S(x,f)=2\pi (-1)^{k/2}x\sum_{n\geq1} \lambda_f(n)\tilde w\Big(\frac{nx}{k^2+\Delta^2}\Big)+\mathcal O \Big(\frac{x\log x}{\Delta}\Big),
\end{equation}
with \(\tilde w\) as given in (\ref{tildewdefrepeated}). Lemma \ref{ibp} also gives the bound \(\tilde w(nx/(k^2+\Delta^2))\ll n^{-2}k^{-1100}\), valid for \(n\geq k^{2+\epsilon/2}/x\geq (k^2+\Delta^2)k^{\epsilon/2}/(2x)\). Consequently, we deduce from (\ref{vorpropexpl}) that
\begin{equation}\label{readytosquare}
\mathcal S(x,f)=2\pi (-1)^{k/2}x\sum_{n\leq k^{2+\epsilon/2}/x} \lambda_f(n)\tilde w\Big(\frac{nx}{k^2+\Delta^2}\Big)+\mathcal O \Big(\frac{x\log x}{\Delta}\Big).
\end{equation}
%This allows us to compute
%\begin{multline}\label{var1}
%\langle \mathcal S(x,f)^2\rangle = 4\pi^2 x^2 \sum_{m,n\leq k^{2+\epsilon/2}/x} \langle \lambda_f(n)\lambda_f(m)\rangle \tilde w\Big(\frac{nx}{k^2+\Delta^2}\Big)\tilde w\Big(\frac{mx}{k^2+\Delta^2}\Big)\\
%+\mathcal O\Big(\Big\langle\Big|x\sum_{n\leq k^{2+\epsilon/2}/x} \lambda_f(n) \tilde w\Big(\frac{nx}{k^2+\Delta^2}\Big)\Big|\cdot \frac{x\log x}{\Delta}\Big\rangle\Big)+\mathcal O \Big(\frac{x^2\log^2 x}{\Delta^2}\Big).
%\end{multline}
Set
\begin{equation}\label{vardef}
\sigma^2=4\pi^2 x^2 \sum_{m,n\leq k^{2+\epsilon/2}/x} \langle \lambda_f(n)\lambda_f(m)\rangle \tilde w\Big(\frac{nx}{k^2+\Delta^2}\Big)\tilde w\Big(\frac{mx}{k^2+\Delta^2}\Big).
\end{equation}
Squaring out (\ref{readytosquare}) and applying Cauchy-Schwarz to the cross-term, we now obtain
\begin{equation}\label{var2}
\langle \mathcal S(x,f)^2\rangle = \sigma^2+\mathcal O \Big(\frac{\sigma x\log x}{\Delta}\Big)+\mathcal O \Big(\frac{x^2\log^2 x}{\Delta^2}\Big).
\end{equation}
%(We have obtained the first error term here by applying Cauchy-Schwarz to the first error term of (\ref{var1}).) 

Since \(x\geq k^{1+\epsilon}\implies k^{2+\epsilon/2}/x\leq k^{1-\epsilon/2}\), we have \(4\pi\sqrt{mn}\leq 999k/1000\) for all integers \(m,n\leq k^{2+\epsilon/2}/x\). 
This means we can apply the Petersson trace formula in the form of Corollary \ref{averages}.
This shows
\[\sigma^2=4\pi^2 x^2\sum_{n\leq k^{2+\epsilon/2}/x}\tilde w\Big(\frac{nx}{k^2+\Delta^2}\Big)^2+\mathcal O\Big(x^2\sum_{m,n\leq k^{2+\epsilon/2}/x}\tilde w\Big(\frac{nx}{k^2+\Delta^2}\Big)\tilde w\Big(\frac{mx}{k^2+\Delta^2}\Big)\exp(-k^{3/5})\Big).\]
Since \(\tilde w(\xi)\ll 1\) for all \(\xi\) (see Lemma \ref{ibp}), the off-diagonal contribution above is \(\mathcal O(e^{-\sqrt k})\). Moreover, using Lemma \ref{diagterms} and the bound \(\tilde w(nx/(k^2+\Delta^2))\ll n^{-2}k^{-1100}\) for \(n\geq k^{2+\epsilon/2}/x\) (as above), one computes
\begin{multline}\label{var2sigmaexp}
4\pi^2 x^2\sum_{n\leq k^{2+\epsilon/2}/x}\tilde w\Big(\frac{nx}{k^2+\Delta^2}\Big)^2=4\pi^2 x^2\sum_{n\geq1}\tilde w\Big(\frac{nx}{k^2+\Delta^2}\Big)^2+\mathcal O(k^{-1000})\\
=x+\mathcal O\Big(\frac x\Delta\Big)+\mathcal O\Big(\frac{x^{3/2}\log^3k}{k}\Big)
\implies \sigma^2=x+\mathcal O\Big(\frac x\Delta\Big)+\mathcal O\Big(\frac{x^{3/2}\log^3k}{k}\Big).
\end{multline}
Choose \(\Delta=k\). Part \ref{variii} of Theorem \ref{thmvar} now follows from (\ref{var2}) and the estimate (\ref{var2sigmaexp}).
\end{proof}

The remainder of this section is devoted to the proof of Lemma \ref{diagterms}. The main idea is to apply the Mellin inversion theorem. To this end, let \(\phi\) denote the Mellin transform of \(\tilde w\):
\begin{equation}\label{phifirstdef}
\phi(s)=\int_0^\infty \xi^{s-1}\tilde w(\xi)d\xi.
\end{equation}
We require the following lemma to control the behaviour of \(\phi\).

\begin{lemma}\label{phibound}
Let \(s=\sigma+it\), and suppose \(\sigma_0\leq\sigma\leq 1-\sigma_0\) for some \(\sigma_0> 0\). In this region, \(\phi\) is holomorphic and satisfies
\[\phi(s)\ll_{A,\sigma_0} (|t|+1)^{-1} (k^2+\Delta^2)^{-\sigma} (k^2+t^2)^{\sigma-1/2}\Big(\frac{\Delta}{|t|+1}\Big)^A,\]
for any integer \(A\geq0\). In particular, if \(|t|\geq \Delta k^\epsilon\), then \(|\phi(s)|\ll |t|^{-2}k^{-10000}\).
\end{lemma}

\begin{proof}
If \(\sigma\geq \sigma_0>0\), then the integral (\ref{phifirstdef}) is absolutely and uniformly convergent (due to the rapid decay of \(\tilde w\), see Lemma \ref{ibp}). Hence \(\phi\) is holomorphic in this region. If additionally \(\sigma\leq \sigma_1<1/4\), we rewrite
\begin{multline*}
\phi(s)=\int_0^\infty \xi^{s-1}\int_0^\infty w(y)J_{k-1}(4\pi\sqrt{(k^2+\Delta^2)y\xi})dy d\xi\\
=\int_0^\infty w(y)\int_0^\infty \xi^{s-1}J_{k-1}(4\pi\sqrt{(k^2+\Delta^2)y\xi})d\xi dy,
\end{multline*}
where interchanging the order of integration is justified by absolute convergence -- the asymptotic (\ref{bessel3trunc}) of Lemma \ref{besselasymptotics} shows \(J_{k-1}(\xi)\ll \xi^{-1/2}\) for \(\xi\geq 2k\). Now making the substitution \(\xi\mapsto (16\pi^2(k^2+\Delta^2)y)^{-1}\xi^2\), we obtain
\begin{multline}\label{phidef}
\phi(s)=2\int_0^\infty w(y)(16\pi^2(k^2+\Delta^2)y)^{-s}\int_0^\infty \xi^{2s-1}J_{k-1}(\xi)d\xi dy\\
=(4\pi^2(k^2+\Delta^2))^{-s}\frac{\Gamma\Big(\frac{k-1}{2}+s\Big)}{\Gamma\Big(\frac{k+1}{2}-s\Big)}\int_0^\infty y^{-s}w(y)dy.
\end{multline}
In the last line, we applied (\ref{mellinbessel}). We have proved (\ref{phidef}) in the region \(0<\sigma_0\leq \sigma\leq \sigma_1<1/4\). But by the identity theorem, (\ref{phidef}) must hold in the larger half-plane \(\sigma\geq \sigma_0\).

We now deduce the result from the expression (\ref{phidef}). First, we bound the gamma factors by an application of Stirling's formula. This shows (see e.g. \cite[proof of Proposition 5.4]{iwanieckowalski}) that
\begin{equation*}
\frac{\Gamma\Big(\frac{k-1}{2}+s\Big)}{\Gamma\Big(\frac{k+1}{2}-s\Big)}\ll (k^2+t^2)^{\sigma-1/2}.
\end{equation*}
So (\ref{phidef}) implies
\begin{equation}\label{gammatermshandled}
\phi(s)\ll (k^2+\Delta^2)^{-\sigma} (k^2+t^2)^{\sigma-1/2}\Big|\int_0^\infty y^{-s}w(y)dy\Big|.
\end{equation}

Integrating by parts \(A\) times, we obtain
\[\int_0^\infty y^{-s}w(y)dy=\frac{1}{s-1}\int_0^\infty y^{-s+1}w'(y)dy=\cdots=\frac{1}{(s-1)\cdots(s-A)}\int_0^\infty y^{-s+A} w^{(A)}(y)dy.\]
Note \(\sigma\leq 1-\sigma_0\implies |s-1|, \ldots, |s-A|\geq \sqrt{t^2+\sigma_0^2}\gg_{\sigma_0} |t|+1\). Using the bound \(w^{(A)}\ll_A \Delta^A\), and noting that provided \(A\geq1\), we are integrating over a set of measure \(2\Delta^{-1}\), we obtain the bound
\[\int_0^\infty y^{-s}w(y)dy\ll_{A,\sigma_0} \Delta^{-1}\Big(\frac{\Delta}{|t|+1}\Big)^A=(|t|+1)^{-1}\Big(\frac{\Delta}{|t|+1}\Big)^{A'}, \text{ for any integer } A'(= A-1)\geq 0.\]
Combining this with (\ref{gammatermshandled}), the lemma is proved.
\end{proof}

To handle some error terms arising in the later computation, we will also require the following standard second moment estimate.

\begin{lemma}\label{zetasecmom}
For \(T\geq 100\), we have the bound
\begin{equation*}
\int_{-T}^T \int_{-T}^{T} \Big|\zeta\Big(\frac12+i(u+v)\Big)\Big|^2\frac{du}{|u|+1}\frac{dv}{|v|+1}\ll \log^3T.
\end{equation*}
\end{lemma}

\begin{proof}
Denote the integral in question by \(\mathcal J\). The change of variables \(w=u+v\) shows
\begin{multline}\label{zetanewvariables}
\mathcal J\coloneqq \int_{-T}^T \int_{-T}^{T} \Big|\zeta\Big(\frac12+i(u+v)\Big)\Big|^2\frac{du}{|u|+1}\frac{dv}{|v|+1}\\
=\int_{-2T}^{2T} \Big|\zeta\Big(\frac12+iw\Big)\Big|^2\int_{\max(-T, w-T)}^{\min(T, w+T)}\Big(\frac{1}{|u|+1}\cdot \frac{1}{|w-u|+1}\Big) dudw.
\end{multline}
To prove the lemma, we shall bound the above integral in the various different regions arising. First splitting the \(w\) integral in (\ref{zetanewvariables}), we obtain
\begin{equation}\label{zetaintwsplit}
\mathcal J=\int_2^{2T}\int_{w-T}^{T}+\int_{-2}^2\int_{\max(-T, w-T)}^{\min(T, w+T)}+\int_{-2T}^{-2}\int_{-T}^{w+T}.
\end{equation}
The middle integral, taken over \(-2\leq w\leq 2\) is 
\[\int_{-2}^2 \Big|\zeta\Big(\frac12+iw\Big)\Big|^2\int_{\max(-T, w-T)}^{\min(T, w+T)}\Big(\frac{1}{|u|+1}\cdot \frac{1}{|w-u|+1}\Big) dudw\ll \int_{-2}^2\int_{-T}^T\Big(\frac{1}{|u|+1}\Big)^2dudw\ll 1.\]
(We used that \(|w|\leq 2\implies |w-u|+1\asymp |u|+1\).) Next consider the last integral (over \(-2T\leq w\leq -2\)) in (\ref{zetaintwsplit}). Making the substitutions \(w\mapsto -w\) and \(u\mapsto -u\), and using that \(|\zeta(s)|=|\zeta(\overline{s})|\), we obtain
\begin{multline*}
\int_{-2T}^{-2} \Big|\zeta\Big(\frac12+iw\Big)\Big|^2\int_{-T}^{w+T}\Big(\frac{1}{|u|+1}\cdot \frac{1}{|w-u|+1}\Big)dudw\\
=\int_2^{2T}\Big|\zeta\Big(\frac12-iw\Big)\Big|^2\int_{-T}^{-w+T}\Big(\frac{1}{|u|+1}\cdot \frac{1}{|w+u|+1}\Big)dudw\\
=\int_2^{2T}\Big|\zeta\Big(\frac12+iw\Big)\Big|^2\int_{w-T}^{T}\Big(\frac{1}{|u|+1}\cdot \frac{1}{|w-u|+1}\Big)dudw, 
\end{multline*}
which is equal to the first integral in (\ref{zetaintwsplit}). It follows from (\ref{zetaintwsplit}) and the two calculations above that
\begin{equation}\label{zetaint3}
\mathcal J=2\int_2^{2T}\Big|\zeta\Big(\frac12+iw\Big)\Big|^2\int_{w-T}^{T}\Big(\frac{1}{|u|+1}\cdot \frac{1}{|w-u|+1}\Big)dudw+\mathcal O(1).
\end{equation}

We next split the inner integral over \(u\) in (\ref{zetaint3}). Since the integrand is non-negative and \(2\leq w\leq 2T\), we use that
\[\int_{w-T}^T\leq \int_{-T}^{2T+1}=\int_{-T}^{-1}+\int_{-1}^1+\int_1^{w-1}+\int_{w-1}^{w+1}+\int_{w+1}^{2T+1}.\]
This implies that the inner integral in (\ref{zetaint3}) (with \(2\leq w\leq 2T\)) is
\begin{align*}
&\int_{w-T}^{T}\Big(\frac{1}{|u|+1}\cdot \frac{1}{|w-u|+1}\Big)du\\
\ll& \int_{-T}^{-1}\frac{1}{-u}\cdot\frac{1}{w-u}du+\int_{-1}^1\frac{1}{w}du+\int_1^{w-1}\frac{1}{u}\cdot\frac1{w-u}du+\int_{w-1}^{w+1}\frac{1}{u}du+\int_{w+1}^{2T+1}\frac{1}{u}\cdot\frac{1}{u-w}du\\
\ll&\frac{1}{w}\Big\{\int_1^T\frac{du}{u}+1+\int_{1}^{w-1}\Big(\frac1u+\frac1{w-u}\Big)du+1+\int_{w+1}^{2T+1}\Big(\frac1{u-w}-\frac1u\Big)du\Big\}\ll \frac{\log T}w.
\end{align*}
Consequently, (\ref{zetaint3}) shows 
\begin{equation}\label{zetaint4}
\mathcal J\ll \log T \int_2^{2T} \Big|\zeta\Big(\frac12 +iw\Big)\Big|^2\frac{dw}w+1.
\end{equation}

Finally, it is a simple matter to bound the above integral. Denote the second moment of \(\zeta\) by
\[\mathcal I(\xi)\coloneqq \int_0^\xi \Big|\zeta\Big(\frac 12+it\Big)\Big|^2dt.\]
It is well-known that \(\mathcal I(\xi)\sim \xi\log\xi\) as \(\xi\to\infty\) (see e.g. \cite[Theorem 7.3]{titchmarsh}). Integrating by parts, it follows
\begin{multline*} 
\int_2^{2T} \Big|\zeta\Big(\frac12 +iw\Big)\Big|^2\frac{dw}w=\int_2^{2T} \frac{\mathcal I'(w)}{w}dw=\frac{\mathcal I(2T)}{2T}-\frac{\mathcal I(2)}{2}+\int_2^{2T}\frac{\mathcal I(w)}{w^2}dw\\
\ll \log T+\int_2^{2T}\frac{\log w}w dw\ll \log^2T.
\end{multline*}
The lemma now follows from (\ref{zetaint4}).
\end{proof}

\begin{proof}[Proof of Lemma \ref{diagterms}]
Throughout this proof, we denote \(\mathfrak c\coloneqq k^2+\Delta^2\) for simplicity. We apply the Mellin inversion theorem:
\[\tilde w\Big(\frac{nx}{\mathfrak c}\Big)=\frac{1}{2\pi i}\int_{(2/3)} \Big(\frac{\mathfrak c}{nx}\Big)^{s}\phi(s)ds.\]
It follows
\begin{equation}\label{sum1}
\sum_{n\geq1}\tilde w\Big(\frac{nx}{\mathfrak c}\Big)^2=\Big(\frac{1}{2\pi i}\Big)^2\int_{(2/3)}\int_{(2/3)}\Big(\frac{\mathfrak c}{x}\Big)^{z+w}\zeta(z+w)\phi(z)\phi(w)dzdw.
\end{equation}
We now shift the contours of integration to the line
%\footnote{This turns out to be the optimal choice: the \(\mathfrak c/x\) terms decrease as we shift to the left, but once we shift past \(1/4\) the \(\zeta\) terms begin to increase at an overwhelming rate.} 
\(\sigma=1/4\) (note the contributions of all horizontal contours are negligible since \(\phi(s)\) is negligible whenever \(|t|\geq \Delta k^{\epsilon}\)). The only pole arises at \(z+w=1\), and so it follows
\begin{multline}\label{sum2}
\sum_{n\geq1}\tilde w\Big(\frac{nx}{\mathfrak c}\Big)^2=\frac{\mathfrak c}{x}\frac{1}{2\pi i}\int_{(2/3)}\phi(z)\phi(1-z)dz\\
+\Big(\frac{1}{2\pi i}\Big)^2\int_{(1/4)}\int_{(1/4)}\Big(\frac{\mathfrak c}{x}\Big)^{z+w}\zeta(z+w)\phi(z)\phi(w)dzdw.
\end{multline}
We now bound the second integral. Write \(\Im z=u\) and \(\Im w=v\). Using the bound \(\phi(\sigma+it)\ll |t|^{-2}k^{-10000}\) (valid for \(|t|\geq \Delta k^\epsilon\)) of Lemma \ref{phibound} and the convexity bound \(\zeta(1/2+it)\ll (|t|+1)^{1/4+\epsilon}\) (see e.g. \cite[(5.1.5)]{titchmarsh}), we observe
\begin{multline*}
\Big(\frac{1}{2\pi i}\Big)^2\int_{(1/4)}\int_{(1/4)}\Big(\frac{\mathfrak c}{x}\Big)^{z+w}\zeta(z+w)\phi(z)\phi(w)dzdw\\
=\frac{1}{4\pi^2} \Big(\frac{\mathfrak c}{x}\Big)^{1/2}\int_{-\Delta k^\epsilon}^{\Delta k^\epsilon}\int_{-\Delta k^\epsilon}^{\Delta k^\epsilon}\zeta\Big(\frac 12+i(u+v)\Big)\Big(\frac{\mathfrak c}{x}\Big)^{i(u+v)}\phi(1/4+iu)\phi(1/4+iv)dudv\\
+\mathcal O\Big(\Big(\frac{\mathfrak c}{x}\Big)^{1/2}\int_{-\infty}^\infty \Big\{\int_{-\infty}^{-\Delta k^\epsilon}+ \int_{\Delta k^\epsilon}^\infty\Big\} (|u+v|+1)^{1/4+\epsilon} |u|^{-2}k^{-10000}|\phi(1/4+iv)|dudv\Big).
\end{multline*}
The error term above is \(\mathcal O(k^{-1000})\). Applying Lemma \ref{phibound} (with \(A=0\)), we therefore obtain
\begin{align} \label{zetaint2}
&\Big(\frac{1}{2\pi i}\Big)^2\int_{(1/4)}\int_{(1/4)}\Big(\frac{\mathfrak c}{x}\Big)^{z+w}\zeta(z+w)\phi(z)\phi(w)dzdw\\
\ll& x^{-1/2}\int_{-\Delta k^\epsilon}^{\Delta k^\epsilon}\int_{-\Delta k^\epsilon}^{\Delta k^\epsilon}\Big|\zeta\Big(\frac12 +i(u+v)\Big)\Big| (k^2+u^2)^{-1/4} (k^2+v^2)^{-1/4}\frac{ du}{|u|+1}\frac{dv}{|v|+1}+k^{-1000}\nonumber\\
\ll& x^{-1/2}k^{-1}\int_{-\Delta k^\epsilon}^{\Delta k^\epsilon}\int_{-\Delta k^\epsilon}^{\Delta k^\epsilon}\Big(1+\Big|\zeta\Big(\frac12 +i(u+v)\Big)\Big|^2\Big)\frac{ du}{|u|+1}\frac{dv}{|v|+1}+k^{-1000}\nonumber\\
\ll& x^{-1/2}k^{-1}\log^3 k,\nonumber
\end{align}
by Lemma \ref{zetasecmom}. Combining (\ref{sum2}) and (\ref{zetaint2}), we have established
\begin{equation}\label{sum3}
\sum_{n\geq1}\tilde w\Big(\frac{nx}{\mathfrak c}\Big)^2=\frac{\mathfrak c}{x}\frac{1}{2\pi i}\int_{(2/3)}\phi(z)\phi(1-z)dz+\mathcal O(x^{-1/2}k^{-1}\log^3k).
\end{equation}

It remains to evaluate the main term above. To do so, we first apply a version of the Plancherel theorem for Mellin transforms (see \cite[§A.2, (A.6)]{ivic}), which gives
\begin{equation}\label{plancherelmellin}
\frac{1}{2\pi i}\int_{(2/3)}\phi(z)\phi(1-z)dz=\int_0^\infty \tilde w(\xi)^2 d\xi.
\end{equation}
The second step involves applying another version of the Plancherel theorem relating the \(L^2\) norm of the Hankel transform \(\tilde w\) to the \(L^2\) norm of \(w\). This is essentially \cite[Lemma 3.4]{lauzhao}. By the definition (\ref{tildewdefrepeated}) we have
\begin{align}\label{needlabel}
\int_0^\infty \tilde w(\xi)^2&=\int_0^\infty \tilde w(\xi) \int_0^\infty w(u) J_{k-1}(4\pi \sqrt{\xi u \mathfrak c}) du d\xi\\
&=\frac{1}{4\pi^2\mathfrak c}\int_0^\infty \tilde w\Big(\frac{\xi^2}{16\pi^2\mathfrak c}\Big)\int_0^\infty w(u^2) J_{k-1}(\xi u) \xi u dud\xi\nonumber\\
&=\frac{1}{4\pi^2\mathfrak c}\int_0^\infty w(u^2)\int_0^\infty \tilde w\Big(\frac{\xi^2}{16\pi^2\mathfrak c}\Big) J_{k-1}(\xi u) \xi u d\xi du.\nonumber
\end{align}
(In the second line, we substituted \(\xi\mapsto \xi^2/(16\pi^2\mathfrak c)\) and \(u\mapsto u^2\).) The inner integral in (\ref{needlabel}) can now be recognised as a Hankel transform of \(\tilde w\). Since \(\tilde w\) is itself a Hankel transform of \(w\), we can relate the inner integral to \(w\) using Hankel inversion. Indeed, we have
\[\tilde w\Big(\frac{\xi^2}{16\pi^2\mathfrak c}\Big)=\int_0^\infty w(v)J_{k-1}(\xi \sqrt{v})dv=\int_0^\infty 2vw(v^2)J_{k-1}(\xi v)dv,\]
which is to say that \(\tilde w(\xi^2/(16\pi^2\mathfrak c))\) is the Hankel transform of the function \(2vw(v^2)\). Applying the Hankel inversion theorem (see \cite[§B.5]{iwaniecspectral}) now shows that the inner integral of (\ref{needlabel}) is
\[\int_0^\infty \tilde w\Big(\frac{\xi^2}{16\pi^2\mathfrak c}\Big)J_{k-1}(\xi u)\xi u d\xi = 2uw(u^2).\]
Replacing this in (\ref{needlabel}), we conclude
\begin{equation*}
\int_0^\infty \tilde w(\xi)^2d\xi =\frac{1}{2\pi^2\mathfrak c}\int_0^\infty uw(u^2)^2 du =\frac{1}{4\pi^2\mathfrak c}\int_0^\infty w(u)^2 du
=\frac{1}{4\pi^2 \mathfrak c}(1+\mathcal O(\Delta^{-1})).
\end{equation*}
Combining the above estimate with (\ref{sum3}) and (\ref{plancherelmellin}), we have established
\begin{equation*}
\sum_{n\geq1}\tilde w\Big(\frac{nx}{\mathfrak c}\Big)^2=\frac{1}{4\pi^2 x}+\mathcal O(x^{-1}\Delta^{-1})+\mathcal O(x^{-1/2}k^{-1}\log^3 k),
\end{equation*}
completing the proof of Lemma \ref{diagterms}.
\end{proof}

\appendix

\section{The Bessel Functions}\label{besselappendix}

In this appendix, we introduce the Bessel functions and record several useful facts. The Bessel functions \(J_{k-1}\) arise in the Petersson trace formula (Lemma \ref{trace}) and the Vorono\"i summation formula (Lemma \ref{vorprop}), so the results from this appendix will be used repeatedly throughout this paper.

We begin with a brief overview. Most of the basic facts can be found in Watson \cite{watson}. We are primarily concerned with Bessel functions of the first kind, which are defined \cite[p.40]{watson} by
\begin{equation}\label{besseldef}
J_\nu(z)=\sum_{l\geq 0}\frac{(-1)^l}{l!\Gamma(\nu+1+l)}\Big(\frac{z}{2}\Big)^{\nu+2l}.
\end{equation}
Throughout this appendix, the index \(\nu\) and argument \(z\) are assumed to be real and positive. In the case that \(\nu=n\) is an integer, we have the integral representation \cite[§2.2]{watson}:
\begin{equation}\label{besselintrep2}
J_n(z)=\frac{1}{\pi}\int_0^\pi \cos(n\theta-z\sin\theta)d\theta=\frac{1}{2\pi}\int_{-\pi}^\pi e^{in\theta}e^{-iz\sin\theta}d\theta.
\end{equation}
One immediately deduces the bound \(|J_n(z)|\leq 1\), valid for all \(z\geq0\). Differentiating (\ref{besseldef}), one can obtain the following useful relation \cite[§3.2]{watson}:
\begin{equation}\label{besseldiff1}
\frac{d}{dz} (z^\nu J_\nu(z))=z^\nu J_{\nu-1}(z).
\end{equation}
We also record that the Mellin transform of \(J_\nu(z)\) is given by \cite[(7.9.2)]{titchmarshfourier}: 
\begin{equation}\label{mellinbessel}
2^{s-1}\frac{\Gamma\Big(\frac{\nu}{2}+\frac{s}{2}\Big)}{\Gamma\Big(\frac{\nu}{2}+1-\frac{s}{2}\Big)}, \text{ valid for } -\nu< \Re s< 3/2.
\end{equation}

In the remainder of this appendix, we record some asymptotics for the Bessel functions \(J_\nu(z)\), valid when \(\nu\) is sufficiently large. The behaviour of \(J_\nu(z)\) is roughly as follows. When the argument \(z\) is small relative to the index \(\nu\), \(J_\nu(z)\) is small. There is a transition when \(z\approx \nu\): when \(\nu-\nu^{1/3}\leq z\leq \nu+\nu^{1/3}\), \(J_\nu(z)\) reaches its global maximum of size \(\approx \nu^{-1/3}\). Finally, once the argument becomes larger than the index, \(J_\nu(z)\) oscillates with amplitude decaying roughly like \(z^{-1/2}\). When \(z\leq \nu\), it is often convenient to make the change of variables \(z=\nu\sech\alpha\), with \(\alpha\in[0, \infty)\). Similarly, when \(z\geq \nu\) we set \(z=\nu\sec\beta\) with \(\beta\in[0,\pi/2)\).

We first record some asymptotics valid uniformly for all \(z\), due to Langer \cite{langer1}, \cite{langer2}. These are primarily useful when \(z\approx \nu\). Firstly, when \(z=\nu\sech\alpha\leq \nu\), we have
\begin{equation}\label{transsmall}
J_\nu(\nu\sech\alpha)=\frac1\pi \Big(\frac{\alpha-\tanh\alpha}{\tanh\alpha}\Big)^{1/2}K_{1/3}(\nu\alpha-\nu\tanh\alpha)+\mathcal O(\nu^{-4/3}),
\end{equation}
where \(K_{1/3}\) denotes the modified Bessel function of the second kind. When \(z=\nu\sec\beta\geq \nu\), we have 
\begin{equation}\label{transbig}
J_\nu(\nu\sec\beta)=\frac{1}{\sqrt{3}}\Big(\frac{\tan\beta-\beta}{\tan\beta}\Big)^{1/2}(J_{1/3}(\nu\tan\beta-\nu\beta)+J_{-1/3}(\nu\tan\beta-\nu\beta))+\mathcal O(\nu^{-4/3}).
\end{equation}
These are (68) and (66) of \cite[§14]{langer1}. They are proved in \cite{langer1} for \(\alpha\) and \(\beta\) which are \(\mathcal O(\nu^{-1/3})\), or equivalently for \(\nu\sech\alpha=\nu-\mathcal O( \nu^{1/3})\) and \(\nu\sec\beta=\nu+\mathcal O(\nu^{1/3})\). However, these are extended to hold uniformly for all \(\alpha, \beta\) in \cite[§19]{langer2}. (The notation in \cite{langer2} differs from that in \cite{langer1}, but one recovers (\ref{transsmall}) and (\ref{transbig}) after changing variables and applying the identities for the modified Bessel functions \(I_{\pm1/3}\) and \(K_{1/3}\) as in \cite[§14]{langer1}.)

We next consider the Bessel function \(J_\nu(z)\) away from the transition regime. If \(|z-\nu|/\nu^{1/3}\) is large, then the arguments \(\nu\alpha-\nu\tanh\alpha\) and \(\nu\tan\beta-\nu\beta\) of the Bessel functions \(K_{1/3}\) and \(J_{\pm 1/3}\) in (\ref{transsmall}) and (\ref{transbig}) are large. Consequently, standard asymptotic expansions for \(K_{1/3}\) and \(J_{\pm1/3}\) are available, which yields an expansion of \(J_\nu\) away from the transition. However, the error term \(\mathcal O(\nu^{-4/3})\) quickly becomes overwhelming.

It is possible to compute the entire asymptotic expansion of \(J_\nu(z)\) away from the transition. 
% When \(z=\nu\sech\alpha\leq \nu-\nu^{1/3}\), Langer \cite[§13, (64)]{langer1} proves the following:
% \begin{equation}\label{langerasympexpandsmall}
% J_\nu(\nu\sech\alpha)\sim \frac{\exp(-\nu(\alpha-\tanh\alpha))}{(2\pi\nu\tanh\alpha)^{1/2}} \Big(1+\sum_{n\geq1}\frac{P_n((\alpha-\tanh\alpha)^{-1})}{\nu^n}\Big), 
% \end{equation}
% where the \(P_n\) are some (in principle computable) polynomials of degree \(n\). 
When \(z=\nu\sec\beta\geq \nu+\nu^{1/3+\epsilon}\), Langer \cite[§13, (63)]{langer1} proves the following\footnote{Langer \cite[§13, (64)]{langer1} also gives a corresponding expansion of \(J_\nu(\nu\sech\alpha)\), valid when \(z=\nu\sech\alpha\leq \nu-\nu^{1/3+\epsilon}\). These asymptotic expansions of \(J_\nu(\nu\sech\alpha)\) (\cite[§13, (64)]{langer1}) and \(J_\nu(\nu\sec\beta)\) ((\ref{langerasympexpandbig})) are essentially due to Debye, who first established them (giving all terms in the expansion) for any fixed \(\alpha, \beta\). See \cite[§8.3-§8.41]{watson} for details.}:
\begin{multline}\label{langerasympexpandbig}
J_\nu(\nu\sec\beta)\sim \Big(\frac{2}{\pi\nu\tan\beta}\Big)^{1/2}\Big\{\cos\Big(\nu\tan\beta-\nu\beta-\frac\pi4\Big)\Big(1+\sum_{n\geq1}\frac{P_n((\tan\beta-\beta)^{-1})}{\nu^n}\Big)\\
+\sin\Big(\nu\tan\beta-\nu\beta-\frac\pi4\Big)\sum_{n\geq1}\frac{Q_n((\tan\beta-\beta)^{-1})}{\nu^n}\Big\},
\end{multline}
where \(P_n\) and \( Q_n\) denote some (in principle computable) polynomials of degree \(n\).

We shall summarise the asymptotics (\ref{transsmall}), (\ref{transbig}) and (\ref{langerasympexpandbig}) in a lemma, giving them in a form convenient for use. But first, we introduce the Airy function. This will allow us to conveniently unify the transition regime asymptotics (\ref{transsmall}) and (\ref{transbig}). One defines the Airy function \(\Ai(x)\) for real \(x\) (see \cite[Ch.2, §8]{olver}) by the improper integral
\[\Ai(x)=\frac1\pi \int_0^\infty \cos(t^3/3+xt)dt.\] 
It is then the case that for \(x\geq0\), 
\begin{equation}\label{airybessel}
\Ai(x)=\frac{x^{1/2}}{\pi\sqrt3} K_{1/3}\Big(\frac23 x^{3/2}\Big), \text{ and } \Ai(-x)=\frac{x^{1/2}}{3}\Big(J_{1/3}\Big(\frac23 x^{3/2}\Big)+J_{-1/3}\Big(\frac23 x^{3/2}\Big)\Big).
\end{equation}
These relations are (1.04) and (1.05) of \cite[Ch.11, §1.1]{olver}. When the argument of the Airy function is large, asymptotic expansions are available. For sufficiently large \(x\geq 0\), one has
\begin{equation}\label{airypos}
\Ai(x)=\frac{1}{2\sqrt{\pi}}x^{-1/4}\exp\Big(-\frac{2}{3} x^{3/2}\Big)(1+\mathcal O(x^{-3/2})),
\end{equation}
and
\begin{equation}\label{airyneg}
\Ai(-x)=\frac{1}{\sqrt{\pi}}x^{-1/4}\cos\Big(\frac{2}{3}x^{3/2}-\frac{\pi}{4}\Big)(1+\mathcal O(x^{-3/2})).
\end{equation}
The derivative, \(\Ai'\), has the expansions (again for sufficiently large \(x\geq0\))
\begin{equation}\label{airy'pos}
\Ai'(x)=-\frac{1}{2\sqrt{\pi}}x^{1/4}\exp\Big(-\frac{2}{3} x^{3/2}\Big)(1+\mathcal O(x^{-3/2})),
\end{equation}
and
\begin{equation}\label{airy'neg}
\Ai'(-x)=\frac{1}{\sqrt{\pi}} x^{1/4}\sin\Big(\frac{2}{3}x^{3/2}-\frac{\pi}{4}\Big)(1+\mathcal O(x^{-3/2})).
\end{equation}
These asymptotics are given in (1.07), (1.08) and (1.09) of \cite[Ch.11, §1.1]{olver}. Being continuous, the Airy function \(\Ai\) is bounded on compact intervals. Consequently, for any real \(x\) we have the bounds
\begin{equation}\label{airybounds}
\Ai(x)\ll \frac{1}{1+|x|^{1/4}}, \text{ and similarly } \Ai'(x)\ll 1+|x|^{1/4}.
\end{equation}
We also note the following fact (see \cite[Ch.11, §12.2]{olver}):
\begin{equation}\label{airyint}
\int_{-\infty}^\infty \Ai(x)dx=1.
\end{equation}

Equipped with this background on the Airy function, we reformulate the asymptotics (\ref{transsmall}), (\ref{transbig}) and (\ref{langerasympexpandbig}) in the following convenient way.

\begin{lemma}[Bessel Function Asymptotics]\label{besselasymptotics}
Let \(\nu>0\) be sufficiently large. For any \(\epsilon>0\) and \(z\geq \nu+\nu^{1/3+\epsilon}\), we have 
\begin{equation}\label{bessel3trunc}
J_\nu(z)=\sqrt{\frac{2}{\pi}}(z^2-\nu^2)^{-1/4}\cos \omega(z)+\mathcal O_\epsilon\Big(\frac{z^2}{(z^2-\nu^2)^{7/4}}\Big),
\end{equation}
where the phase \(\omega\) is given by 
\begin{equation}\label{omega}
\omega(z)=\omega_\nu(z)=(z^2-\nu^2)^{1/2}-\nu\arctan\big((z^2/\nu^2-1)^{1/2}\big)-\pi/4.
\end{equation}
Furthermore, if \(y\) satisfies\footnote{This restriction appears since the main term is of comparable size to the error term for larger values of \(y\).} \(|y|\leq \nu^{4/15}\) we have
\begin{equation}\label{krasikovgood}
J_\nu(\nu+y\nu^{1/3})=\frac{2^{1/3}}{\nu^{1/3}}\Ai(-2^{1/3}y)+\mathcal O \Big(\frac{1+|y|^{9/4}}{\nu}\Big).
\end{equation}
\end{lemma}

\begin{remark}
For convenience, we record
\begin{equation}\label{omega'}
\omega'(z)=\frac{(z^2-\nu^2)^{1/2}}{z},
\end{equation}
and
\begin{equation}\label{omega''}
\omega''(z)=\frac{\nu^2}{z^2(z^2-\nu^2)^{1/2}}.
\end{equation}
\end{remark}

\begin{proof}
The asymptotic (\ref{bessel3trunc}) follows straightforwardly from (\ref{langerasympexpandbig}) upon making the change of variables \(z=\nu\sec\beta\). Indeed, one then has \(\nu\tan\beta=\nu(\sec^2\beta-1)^{1/2}=(z^2-\nu^2)^{1/2}\) and 
\[\tan\beta-\beta=\Big(\frac {z^2}{\nu^2}-1\Big)^{1/2}-\arctan\Big(\Big(\frac{z^2}{\nu^2}-1\Big)^{1/2}\Big)\gg \min \Big\{\Big(\frac{z^2}{\nu^2}-1\Big)^{3/2}, \Big(\frac{z^2}{\nu^2}-1\Big)^{1/2}\Big\}.\]
(The lower bound follows immediately from the bound \(\arctan \xi\leq \pi/2\), valid for all \(\xi\), and the asymptotic \(\xi-\arctan \xi\sim \xi^3/3\), valid for \(\xi=o(1)\).) Consequently (recalling \(P_n\) and \(Q_n\) are some polynomials of degree \(n\)) for \(z\geq \nu+\nu^{1/3+\epsilon}\) we bound
\begin{align*} 
&\cos\Big(\nu\tan\beta-\nu\beta-\frac\pi4\Big)\sum_{n\geq1}\frac{P_n((\tan\beta-\beta)^{-1})}{\nu^n}+\sin\Big(\nu\tan\beta-\nu\beta-\frac\pi4\Big)\sum_{n\geq1}\frac{Q_n((\tan\beta-\beta)^{-1})}{\nu^n}\\
&\ll \sum_{n\geq1}\frac{\min\Big\{\Big(\frac{z^2}{\nu^2}-1\Big)^{3/2}, \Big(\frac{z^2}{\nu^2}-1\Big)^{1/2}\Big\}^{-n}}{\nu^n}\ll \sum_{n\geq 1}\frac1{\nu^n} \Big\{\Big(\frac{z^2}{\nu^2}-1\Big)^{-3n/2}+\Big(\frac{z^2}{\nu^2}-1\Big)^{-n/2}\Big\}\\
&\ll_\epsilon \frac1\nu \Big\{\Big(\frac{z^2}{\nu^2}-1\Big)^{-3/2}+\Big(\frac{z^2}{\nu^2}-1\Big)^{-1/2}\Big\} \ll_\epsilon \frac{\nu^2}{(z^2-\nu^2)^{3/2}}+\frac{1}{(z^2-\nu^2)^{1/2}}\ll_\epsilon \frac{z^2}{(z^2-\nu^2)^{3/2}}.
\end{align*}
Here we required the assumption \(z\geq \nu+\nu^{1/3+\epsilon}\implies (z^2/\nu^2-1)^{1/2}\gg \nu^{-1/3+\epsilon/2}\) to ensure convergence of the series, and to bound \(\nu^2(z^2-\nu^2)^{-3/2}\leq z^2 (z^2-\nu^2)^{-3/2}\) and \((z^2-\nu^2)^{-1/2}\leq z^2(z^2-\nu^2)^{-3/2}\). We thus obtain (\ref{bessel3trunc}) by truncating the asymptotic expansion (\ref{langerasympexpandbig}) after the first term.

The transition regime asymptotic follows from (\ref{transsmall}) and (\ref{transbig}). We first apply the identities given in (\ref{airybessel}) to (\ref{transsmall}) and (\ref{transbig}), from which we obtain
\begin{equation}\label{nusechairy} 
J_\nu(\nu\sech\alpha)=\frac{2^{1/3}3^{1/6}}{\nu^{1/3}}\frac{(\alpha-\tanh\alpha)^{1/6}}{(\tanh\alpha)^{1/2}}\Ai\Big(\Big(\frac{3\nu}{2}(\alpha-\tanh\alpha)\Big)^{2/3}\Big)+\mathcal O(\nu^{-4/3}),
\end{equation}
and 
\begin{equation}\label{nusecairy} 
J_\nu(\nu\sec\beta)=\frac{2^{1/3}3^{1/6}}{\nu^{1/3}}\frac{(\tan\beta-\beta)^{1/6}}{(\tan\beta)^{1/2}} \Ai\Big(-\Big(\frac{3\nu}{2}(\tan\beta-\beta)\Big)^{2/3}\Big)+\mathcal O(\nu^{-4/3}).
\end{equation}
First addressing (\ref{nusechairy}), we make the change of variables \(\nu\sech\alpha=\nu+y\nu^{1/3}\) (with \(-\nu^{4/15}\leq y\leq0\)). One computes
\[\tanh\alpha=(1-\sech^2\alpha)^{1/2}=2^{1/2}|y|^{1/2}\nu^{-1/3}+\mathcal O(|y|^{3/2}\nu^{-1}),\]
\[\alpha-\tanh\alpha=\artanh(\tanh\alpha)-\tanh\alpha=\frac13 \tanh^3\alpha+\mathcal O(\tanh^5\alpha)=\frac{2^{3/2}}{3}|y|^{3/2}\nu^{-1}+\mathcal O(|y|^{5/2}\nu^{-5/3}),\]
and
\[\Big(\frac{3\nu}{2}(\alpha-\tanh\alpha)\Big)^{2/3}=2^{1/3}|y| +\mathcal O(|y|^2\nu^{-2/3})=-2^{1/3}y +\mathcal O(|y|^2\nu^{-2/3}).\]
By the mean value theorem and (\ref{airybounds}), it follows 
\[\Ai\Big(\Big(\frac{3\nu}{2}(\alpha-\tanh\alpha)\Big)^{2/3}\Big)=\Ai(-2^{1/3}y)+\mathcal O(|y|^2(1+|y|^{1/4})\nu^{-2/3}).\]
Combining the above calculations with (\ref{nusechairy}) and the Airy function bound (\ref{airybounds}), we obtain
\[J_\nu(\nu\sech\alpha)=J_\nu(\nu+y\nu^{1/3})=\frac{2^{1/3}}{\nu^{1/3}}\Ai(-2^{1/3}y)+\mathcal O\Big(\frac{1+|y|^{9/4}}{\nu}\Big),\]
which proves (\ref{krasikovgood}) in the case \(-\nu^{4/15}\leq y\leq0\).

In the case \(0\leq y\leq \nu^{4/15}\), we instead use (\ref{nusecairy}) and the change of variables \(\nu\sec\beta=\nu+y\nu^{1/3}\). Similarly, one computes
\[\tan\beta=2^{1/2}y^{1/2}\nu^{-1/3}+\mathcal O(y^{3/2}\nu^{-1}),\]
\[\tan\beta-\beta=\frac{2^{3/2}}{3}y^{3/2}\nu^{-1}+\mathcal O(y^{5/2}\nu^{-5/3}),\]
and
\[\Big(\frac{3\nu}{2}(\tan\beta-\beta)\Big)^{2/3}=2^{1/3}y+\mathcal O(y^2\nu^{-2/3}).\]
Exactly as in the case \(y\leq 0\), one combines these calculations and the Airy function bounds (\ref{airybounds}) with (\ref{nusecairy}) to obtain 
\[J_\nu(\nu\sec\beta)=J_\nu(\nu+y\nu^{1/3})=\frac{2^{1/3}}{\nu^{1/3}}\Ai(-2^{1/3}y)+\mathcal O\Big(\frac{1+|y|^{9/4}}{\nu}\Big),\]
which establishes (\ref{krasikovgood}) in the full range \(-\nu^{4/15}\leq y\leq \nu^{4/15}\).
\end{proof}

% Finally, we summarise several bounds for the Bessel functions \(J_\nu\) in the following lemma. These will be deduced from the asymptotics above.

Finally, we give several bounds for the Bessel functions. These are conveniently simplified versions of bounds given by Rankin \cite[§4]{rankin}, which were deduced (in part) from the asymptotics of Langer \cite{langer1}.

\begin{lemma}[Bessel Function Bounds]\label{besselbounds}
Let \(\nu>0\) be sufficiently large. We have the following bounds.

\begin{enumerate}[label=(\roman*)] 
\item \label{besi} If \(0\leq z\leq (\nu+1)/4\), then
\begin{equation}\label{bessboundsmallarg}
J_\nu(z)\ll z^2\exp\Big(-\frac{14\nu}{13}\Big).
\end{equation}

\item \label{besii} If \(0\leq z\leq (\nu+1)-(\nu+1)^{1/3+\delta}\) for some \(0\leq\delta\leq 2/3\), then
\begin{equation}\label{bessel1}
J_\nu(z)\ll \exp(-\nu^\delta).
\end{equation}

\item \label{besiii} Uniformly for \(z\geq0\), we have
\begin{equation}\label{bessel2}
J_\nu(z)\ll \nu^{-1/3}.
\end{equation}

\end{enumerate}
\end{lemma}

\begin{remark}
Whilst both bounds \ref{besi} and \ref{besii} are valid when in the region where \(J_\nu\) is very small, they are useful in different ranges. The first bound \ref{besi} is useful when \(z\) is extremely small (in fact, the somewhat artificial \(z^2\) term is inserted to conveniently ensure convergence of series like \(\sum_m J_\nu(\nu/m)\)). The second bound \ref{besii} is useful when \(z\) is larger, and close to the transition regime. The bound \ref{besiii} is primarily useful in the transition regime (i.e. for \(z\) such that \(|z-\nu|\ll \nu^{1/3}\)), where it is sharp. 
\end{remark}

\begin{proof}
Part \ref{besi} follows from \cite[Lemma 4.1]{rankin}, which gives that for any \(z\geq0\)
\begin{multline*}
J_\nu(z)\ll \nu^{-1/2}\Big(\frac{ez}{2\nu}\Big)^\nu=\nu^{-1/2}\Big(\frac{ez}{2\nu}\Big)^2\exp\Big(-(\nu-2)\log\Big(\frac{2\nu}{ez}\Big)\Big)\\
\ll z^2\nu^{-5/2}\exp\Big(-(\nu-2)\log \Big(\frac{2(\nu+1)}{ez}\Big)\Big).
\end{multline*}
If we assume \(z\leq (\nu+1)/4\iff \log(2(\nu+1)/(ez))\geq 3\log2-1\), it follows
\begin{equation*}
J_\nu(z)\ll z^2\nu^{-5/2}\exp(-(\nu-2)(3\log2-1))\ll z^2\exp\Big(-\frac{14\nu}{13}\Big),
\end{equation*}
as claimed.

Part \ref{besii} can be deduced from \cite[Lemma 4.2]{rankin}. This gives that for \(0\leq z\leq \nu(1-\nu^{-2/3})^{1/2}\),
\begin{equation}\label{rankinlem42}
J_\nu(z)\ll (\nu^2-z^2)^{-1/4}\exp\Big(-\frac{\nu}{3}\Big(1-\frac{z^2}{\nu^2}\Big)^{3/2}\Big)\ll \nu^{-1/3}\exp\Big(-\frac{(\nu-z)^{3/2}}{3\nu^{1/2}}\Big).
\end{equation}
Note that for sufficiently large \(\nu\) and \(0\leq\delta\leq 2/3\), we have \((\nu+1)-(\nu+1)^{1/3+\delta}\leq \nu-\nu^{1/3+\delta}/3\leq \nu(1-\nu^{-2/3})^{1/2}\). 
Our assumption \(z\leq (\nu+1)-(\nu+1)^{1/3+\delta}\) therefore implies \(z\leq \nu(1-\nu^{-2/3})^{1/2}\) and \(\nu-z\geq \nu^{1/3+\delta}/3\).
Consequently, it follows from (\ref{rankinlem42}) that \(J_\nu(z)\ll \nu^{-1/3}\exp(-3^{-5/2} \nu^{3\delta/2})\ll \exp(-\nu^\delta)\) as claimed.

Part \ref{besiii} (cf. \cite[Lemma 4.4]{rankin}) is immediately established by (\ref{rankinlem42}) if \(z\leq \nu-\nu^{3/5}\). If \(\nu-\nu^{-3/5}\leq z\leq \nu+\nu^{3/5}\), then the bound required in \ref{besiii} follows immediately from (\ref{krasikovgood}) and the Airy function bound (\ref{airybounds}). Finally, if \(z\geq \nu+\nu^{3/5}\), the asymptotic (\ref{bessel3trunc}) shows \(J_\nu(z)\ll (z^2-\nu^2)^{-1/4}\ll \nu^{-2/5}\ll \nu^{-1/3}\). 
\end{proof}

\bibliographystyle{plain}
\bibliography{References}

\end{document}